\documentclass{amsart}

\usepackage{these}
\usepackage[all]{xy}
\usepackage[latin1]{inputenc}
\usepackage{amsmath,amssymb,amsfonts,amsthm}
\usepackage[T1]{fontenc}
\usepackage{stmaryrd}
\usepackage{enumerate}
\usepackage{comment}
\usepackage{multirow}
\usepackage{array}
\usepackage{graphicx}
\usepackage{caption}
\usepackage{subcaption}
\usepackage{tikz}

\usepackage{float}
\usepackage{hyperref}
\usetikzlibrary{patterns}
\usetikzlibrary{arrows,shapes,positioning}
\usetikzlibrary{decorations.markings}
\tikzstyle arrowstyle=[scale=1]
\tikzstyle directed=[postaction={decorate,decoration={markings,
    mark=at position 0.6 with {\arrow[arrowstyle]{stealth}}}}]

\renewcommand{\H}{\mathcal{H}}

\title{Dimensions in non-Archimedean geometries}

\author{Florent Martin}
\address{Florent Martin, Laboratoire Paul Painlev\'{e}, Universit\'{e} Lille 1, 59655 Villeneuve d'Ascq, France.}
\email{florent.martin@math.univ-lille1.fr}

\thanks{ The research leading to these results has received funding from 
the European Research Council under the European Community's Seventh 
Framework Programme (FP7/2007-2013) / ERC Grant Agreement n$^o$ 246903 "NMNAG 
and from the
Labex CEMPI (ANR-11-LABX-0007-01)
}
\begin{document}

\date{\today}

\begin{abstract}
Let $K$ be an algebraically closed non-Archimedean field. 
Leonard Lipshitz has introduced a manageable notion of subanalytic sets $S$ of the polydisc $\Kon$ which 
contains affinoid sets and is stable under projection. 
We associate to a subanalytic set $S \subset \Kon$ its counterpart $S_\Berko$, a subset of the Berkovich polydisc. 
This allows us to give a new insight to the dimension of subanalytic sets using the degrees of the completed residual fields. 
With these methods we obtain new results, such as the invariance of the dimension under subanalytic bijection in any characteristic.
Then we study more generally subsets $S \subset K^m\times \Gamma^n$ and 
$S \subset K^m\times \Gamma^n \times k^p$ where $\Gamma$ is the value group and $k$ the residue field. 
We allow $S$ to be either definable in ACVF, or definable in the analytic language of L. Lipshitz.
We define a dimension for such sets $S$. In the case when 
$S \subset K^n$ (resp. $S\subset \Gamma^n$, $S\subset k^n$), it coincides with the above dimension (resp. the o-minimal dimension, the Zariski dimension). 
We prove that this dimension is invariant under definable bijection and decreases under projection.
This allows us to generalize previous results on tropicalization of Berkovich spaces and to place them in a general framework.
\end{abstract}
\maketitle

\tableofcontents

\maketitle


\section{Introduction}
A non-Archimedean field is a field $K$ equipped with a non-Archimedean norm 
$| \cdot | : K \to \R_+$ for which it is complete. \par 
We will also consider the value group $\Gamma \subset \R_+^*$  and the residue field $k$ of our valued field $K$, 
and will use some quantifier elimination results in this context. 
We will set $\Gamma_0 := \Gamma \cup \{0\}$ and we extend the operations $\cdot$ and $<$ from $\Gamma$ 
to $\Gamma_0$ in the following way: $0 <x$ for all $x\in \Gamma$, and $0\cdot x =0$ for all $x\in \Gamma_0$. 

\subsection{Motivations from Tropical Geometry}
Let $K$ be a non-Archimedean algebraically closed field, $\Gamma$ its value group (we set 
$|\cdot |: K \to \Gamma_0$ its norm). 
The process of tropicalization starts with an object $X$ defined over $K$, (for instance 
an algebraic subvariety of $\Gm^n$), 
and associates to it a simpler combinatorial object defined over 
$\Gamma$. 
This process relies on the \emph{tropicalization map}:
\[ 
\begin{array}{lccc}
\Trop :&  (K^*)^n &\to &\Gamma^n \\ 
 & (x_1,\ldots,x_n) & \mapsto &  (|x_1|,\ldots,|x_n|).
  \end{array}
  \]
  
For instance  when $X$ is a subvariety of 
$\Gm^n$, we can study $\Trop(X)$ which is a subset of $\Gamma^n$. Some properties of $X$ reflect in $\Trop (X)$.
In the context of Berkovich spaces (also called $K$-analytic spaces) we will also denote by $\Trop$ the map
 
\[ 
\begin{array}{lccc}
\Trop : (\mathbb{A}_{m,K}^n)^{\an}  &\to &(\R_+)^n \\ 
 & x & \mapsto &  (|T_1(x)|,\ldots,|T_n(x)|)
  \end{array}
  \]
where the $T_i$'s are the coordinate functions.
\par   
We will say that a set 
 $S \subset (\R^*_+)^n$  is a \emph{$\Gamma$-rational polyhedron} if it is a finite intersection of sets of the form 
 \[ \{ (\gamma_1,\ldots,\gamma_n) \in (\R^*_+)^n \st 
 \prod_{i=1}^n \gamma_i^{a_i} \leq \lambda \} 
\]
where $a_i \in \Z$ and $\lambda \in \Gamma$.
A \emph{polytope} is a compact polyhedron. 
A \emph{$\Gamma$-rational polyhedral set} is a finite union of $\Gamma$-rational polyhedra. 
\par 
We state some known results: \index{polyhedral set} 

\begin{theor} \cite[Theorem A]{BieGro},\cite[2.2.3]{EKL}
Let $X$ be an irreducible algebraic subvariety  of 
$\Gm^n$ of dimension $d$. Then $\overline{\Trop(X(K))} \subset (\R_+^*)^n$ is a connected $\Gamma$-rational polyhedral set of pure dimension $d$. 
\end{theor}

\begin{theor}
\cite[Th. 1.1]{Gub07}
Let $X$ be an irreducible closed analytic subvariety of 
$(\Gm^n)^\an$ of dimension $d$. 
Then $\Trop(X) \subset (\R_+^*)^n$ is a locally finite union of $\Gamma$-rational
$d$-dimensional polytopes.
\end{theor}

\begin{theor}
\cite[3.30]{DucIma} and \cite[Th 3.2]{Ducpo}.
Let $X$ be a compact $K$-analytic space of dimension $d$ and 
$f:X \to (\Gm^n)^\an$ an analytic map. 
Then $|f|(X) \subset (\R_+^*)^n$  
is a finite union of $\Gamma$-rational polytopes of dimension $\leq d$.
\end{theor}

The initial motivation of this work was to obtain the following generalization of Antoine Ducros's result \cite[Theorem 3.2]{Ducpo}:

\begin{theor}\ref{theoimprovduc}
Let $X$ be a compact $K$-analytic space of dimension $d$, 
$f : X \to \mathbb{A}^{n,an}_K$ be an analytic map. Then 
$|f|(X) \cap (\R_+^*)^n$ is a $\Gamma$-rational polyhedral set of $(\R_+^*)^n$ of dimension 
$\leq d$.
\end{theor}
It happens that this result is more or less equivalent to the more \emph{rigid-geometry} style result:

\begin{theor}\ref{theoimprovduc} 
Let $X$ be a compact $K$-analytic space of dimension $d$, 
$f : X \to \mathbb{A}^{n,an}_K$ be an analytic map. Then 
$|f|(X(K)) \cap (\Gamma)^n$ is a rational polyhedral set of $(\Gamma)^n$ of dimension 
$\leq d$.
\end{theor}
Where we say that $S\subset \Gamma^n$ is a polyhedron if it is 
 a finite intersection of sets of the form 
 \[ \{ (\gamma_1,\ldots,\gamma_n) \in \Gamma^n \st 
 \prod_{i=1}^n \gamma_i^{a_i} \leq \lambda \} 
\]
where $a_i \in \Z$ and $\lambda \in \Gamma$.
The reader interested by this result can easily proceed to its proof. 
Let us mention that the result \cite[Th 3.2]{Ducpo} has been motivated by the recent development
by Antoine(s) Chambert-Loir and Ducros of a theory of real differential forms on Berkovich spaces \cite{ChaDuc}.

\subsection{Definable sets}
To prove the theorem \ref{theoimprovduc}, we will use the language $\Ldan$ introduced by 
Leonard Lipshitz in \cite{Li_rig}. 
In fact, in  theorem \ref{theoimprovduc}, it is not difficult to
prove that $|f|(X)$ is a polyhedral set using 
the quantifier elimination theorem proved by Leonard Lipshitz \cite[theorem 3.8.2]{Li_rig}. 
But the bound on the dimension in theorem \ref{theoimprovduc} does not follow obviously.
In the rest of the introduction, we will consider only subsets $S\subset \KmGn$ (in particular, 
$\Gamma \subset \R^*_+$ can be a very small subgroup). 
\par 
The list of results mentioned above \cite{BieGro,Gub07,Ducpo} 
which all contain a bound on the dimension of $\Trop(X)$ was a strong evidence in favor of the bound in theorem \ref{theoimprovduc}. 
Once 
one tries to prove it, one realizes that we have a set 
$X \subset K^m$, a map 
$f : X \to \Gamma^n$ and we would like to say that 
\begin{equation}
\label{compdim}
\dim(f(X) ) \leq \dim(X).
\end{equation}
This inequality should be satisfied in any good theory of
dimension. However, here the situation is quite particular, because in 
\eqref{compdim}, we are dealing with two different kinds of dimension.
On the left side, this 
is the very combinatorial dimension of a polyhedral set of  
$\Gamma^n$,
 and on the right side, this a dimension for subsets of $K^n$.
What formula 
\eqref{compdim} suggests is that these two dimensions should interact nicely. 
It is then natural to wonder whether it is possible to obtain a \emph{unified} theory of dimension, where 
\eqref{compdim} would be a simple corollary.\par
More precisely, assume that we could assign a dimension 
$\dim(S) \in \N$ to each nonempty subset $S$ of $K^m \times \Gamma^n$ such that

\begin{equation}
\label{Dim}
\small{
\begin{array}{|c|c|}
\hline
\multicolumn{2}{|c|}{ \text{ \textsc{Dimension axioms}} }  \\
\hline
\text{Dim} \ 0& \dim(S) =0  \text{if and only if}  \ S \ \text{is finite}. \\
\hline
\text{Dim} \ 1& \text{When} \ S \subset K^m, \ 
  \dim(S) \ \text{is some reasonable dimension of} \ S \ \text{as a subset of} \ K^m. \\
\hline
\text{Dim} \ 2 & \text{When} \ S \subset \Gamma^n, \  \dim(S) \ \text{is a 
reasonable dimension of} \ S \ \text{as a subset of} \ \Gamma^n.  \\ 
\hline
 \text{Dim} \ 3 & 
 \text{If} \ S \subset \KmGn ,
T \subset K^{m'} \times \Gamma^{n'}
\ \and \ f : S \to T \\
& \text{is a  map, then} \  \dim(f(S)) \leq \dim(S). \\
\hline
 \text{Dim} \ 3' & 
 \text{If} \ S \subset \KmGn ,
T \subset K^{m'} \times \Gamma^{n'}
\ \and \ f : S \to T \\
& \text{is a  bijection, then} \  \dim(f(S)) = \dim(S). \\
\hline 
\text{Dim} \ 4 & \text{If} \ S,T \subset \KmGn, \ \dim(S\cup T) = \max(\dim(S),\dim(T)) \\
\hline
\text{Dim} \ 5 & \text{If} \ S \subset \KmGn ,
T \subset K^{m'} \times \Gamma^{n'}, \ \text{then} \ \dim(S\times T) = \dim(S)+\dim(T) \\
\hline
\end{array}
}
\end{equation}

Then \eqref{compdim} would be a simple consequence of this dimension theory.
Of course, in such a generality this is an unattainable goal. 
In order to build such a theory, 
one has to work with a \emph{restricted class} of subsets of 
$K^m \times \Gamma^n$.  and a restricted class of maps.\par 
To do this, there are two 
\emph{restricted classes} that we will consider in this text:
\begin{itemize}
\item The class of definable subsets $S \subset \KmGn$, 
in the first order theory of algebraically valued fields\footnote{In that case $K$ is an arbitrary algebraically closed valued field, 
so $\Gamma$ can be an arbitrary totally ordered abelian  divisible group not necessarily of height $1$.}. 
\item 
If $K$ is an algebraically closed non-Archimedean field, the class of subanalytic sets $S \subset \Kom \times \Gamma^n$ as defined by L. Lipshitz in \cite{Li_rig}.  
\end{itemize}
The main results of this text is that for these two classes we can build a good dimension theory, as described in \eqref{Dim}: 
\begin{theor} \ref{theodimmN}
There exists a dimension theory which satisfies the  axioms listed in \eqref{Dim} in 
the following two cases.

\begin{enumerate}
\item When $K$ is an algebraically closed valued field and we restrict  to 
definable subsets $S \subset \KmGn$ in the theory ACVF and definable maps 
$f: S \to T$ (which means that the graph of $f$ is definable). 
\item When $K$ is an algebraically closed non-Archimedean field and we restrict 
to subanalytic sets $S\subset \KmGn$ as defined in \cite{Li_rig}, and subanalytic maps between them.
\end{enumerate}

\end{theor}

It appears however that assigning a dimension $\dim(S) \in \N$ to a set 
$S\subset \KmGn$ is quite coarse.
\begin{exem}
For instance we must have $\dim(K^*) = \dim(\Gamma)=1$. 
Moreover, the
map $|\cdot | : K^* \to \Gamma$ is surjective, which is 
compatible with the fact that $\dim(K^*) \geq \dim(\Gamma)$. 
At the opposite, if 
$f : \Gamma \to K$ is a definable map (either in the theory of algebraically closed valued fields, 
or in the theory of subanalytic sets), then it must have a finite image (see lemma \ref{finite_image}), so 
$\dim(f(\Gamma))=0$. 
\end{exem}
This example illustrates a general philosophy that we can sum up in:
\begin{center}
\label{philo}
\framebox{One dimension of $K$ is stronger than one dimension of $\Gamma$.}
\end{center}

\begin{exem}
\label{ex2}
Let $S = K \coprod \Gamma^2 $. 
If one wants to give a precise sense to this, one can take 
\[S = \big( \{x\}\times \Gamma^2 \big) \coprod \big( K \times \{ (\gamma_1,\gamma_2) \} \big) \subset K \times \Gamma^2 \]
with $x\in K$ and $(\gamma_1,\gamma_2) \in \Gamma^2$. 
According to the dimension theory built in 
theorem \ref{theodimmN}, $\dim(S)=2$. However, one 
can check that there is no definable map
$f : S \to K^n$ such that 
$\dim(f(S)) =2$. Hence in that case, saying that $S$ has 
dimension $2$ seems to hide a part of the situation.\\
More generally, 
we should have $\dim(\KmGn) = m+n$. 
Now, there exist some definable maps 
$f : \KmGn \to \Gamma^{m+n}$ whose image has dimension $m+n$, but there does not exist a 
definable map $g: \KmGn \to K^{m+n}$ whose image has 
dimension $m+n$. 
Actually, for such a map $g$, it is natural to conjecture that 
$\dim(\im(g)) \leq m$. This is true, and will follow for 
instance from theorem \ref{theodimmixed}. 
\end{exem}
Hence we should refine the leitmotiv given in \eqref{philo}, and 
replace it by \vspace{3pt}
\begin{center}
 \framebox{One dimension of $K$ can be converted in one dimension of $K$ or one dimension of $\Gamma$.} \\
\framebox{One dimension of 
$\Gamma$ can not be converted in one dimension of $\Gamma$.}
\end{center}
In order to take this into account, one  should
look for a dimension theory that would be two-dimensional: 
to a subset $S \subset \KmGn$ we would assign a dimension 
$\dim(S)=(d_1,d_2) \in \N^2$ with $d_1$ 
for the part of the valued field $K$ and $d_2$ for the valued group $\Gamma$. 
However this still does not work. 
Indeed, what should be the dimension 
of $S = K \coprod \Gamma^2 $ that we have seen in example \ref{ex2}? 
We should hesitate between 
$(1,0)$ and $(0,2)$. 
According to the above leitmotiv, it would be unfair to 
throw $(1,0)$, but on the other hand since $0+2 > 1+0$, we do not want neither to throw 
$(0,2)$. The only solution seems to keep both of them, and to define the dimension 
of a definable subset of $S\subset \KmGn$ as a finite subset of $\N^2$.\par 
In order to obtain a simple statement, we must specify a few things. 
On $\N^2$ we will consider the partial order $\leq$ defined by
\[(a,b) \leq (a',b') \Leftrightarrow a\leq b \ \and \ a'\leq b' . \]
We will say that a set $A\subset \N^2$ 
is a \textbf{lower set} if it is stable under $\leq$.
If 
$(a,b) \in \N^2$, we set 
\[\langle (a,b) \rangle := \{ (x,y)\in \N^2 \st (x,y) \leq (a,b) \}. 
\]
This is the smallest lower set of $\N^2$ containing $(a,b)$. 
For aesthetic reasons, if $A,B \subset \N^2$ we will set
$\max(A,B):=A \cup B$. \par 
Let then $K$ be a non-Archimedean algebraically closed field or an algebraically closed valued field, 
and let \emph{definable} meaning subanalytic or 
definable in ACVF depending on this.

\begin{theor} \ref{theodimmixed}
There exists  a dimension  which assigns to each definable
set $S \subset \KmGn$  a dimension 
$\dim (S) \subset \N^2$ which is a finite lower set of $\N^2$, satisfying the following properties: 
\begin{enumerate}[(i)]
\item 
If $S \subset \Gamma^n$, then 
$\dim(S) = \langle (0,d) \rangle$ where $d$ is the classical dimension of 
$S$ as a subset of $\Gamma^n$ as defined in section \ref{sec1.4}. 
\item If $S \subset K^m$ is definable, then 
$\dim(S) = \langle (d,0) \rangle$ where $d$ is the classical dimension of 
$S$ as 
exposed in sections \ref{sec1.2} and \ref{subsection:dimACVF}.
\item 
If $f : S \to T$ is a definable map, then
\[ \dim (f(S) ) \leq \max_{k\geq 0} \big( \dim(S) + (-k,k) \big). \]
\item 
If $f: S \to T$ is a definable  bijection, then $\dim(S)=\dim(T)$. 
\item $\dim(S \times T) = \dim(S)+ \dim(T)$.
\end{enumerate}
\end{theor}
We have used 
the convention that  if $A\subset \N^2$ and $B \subset \Z^2$ then 
$A+B :=\{a+b \st a\in A, \ b\in B \}\cap \N^2$.  \par 
We want to point out that this theorem implies theorem \ref{theodimmN}. 
Indeed, for a definable set 
$S \subset \KmGn$, if we set 
\[\dim_\N (S) := \max_{(d_1,d_2) \in \dim(S) } d_1 +d_2 \]
the properties listed in 
theorem \ref{theodimmixed} above imply that $\dim_\N$ fulfills the properties of theorem \ref{theodimmN}.
This is actually  the way we do things: 
we first prove theorem \ref{theodimmixed}, and obtain theorem \ref{theodimmN} 
as a corollary.

\subsection{Berkovich points as a tool for dimension}

In the analytic setting, i.e. when $K$ is a non-Archimedean field, the classical dimension of 
a subanalytic set $S\subset \Kon$ refereed to above has been introduced by 
Leonard Lipshitz and Zachary Robinson in \cite{LRdim}: the dimension $\dim(S)$ of a subanalytic set is the 
greatest integer $d$ for which there exists a coordinate projection 
$\pi : K^n \to K^d$ such that 
$\pi(S)$ has nonempty interior. 
Many properties are proved in \cite{LRdim}, let us mention for instance a smooth stratification theorem. 
However, to give an example, a formula such as 
$\dim(S_1\cup S_2) = \max (\dim(S_1), \dim(S_2) )$ does not follow easily.
In addition, 
we did not manage to deduce from the results of \cite{LRdim} the invariance of the dimension 
under subanalytic bijection, and the fact that dimension decreases under subanalytic map. \par 
Berkovich spaces offer a nice view point on the question.  
If $S\subset \Kon$ is subanalytic, we attach to it (see definition \ref{defi:sub_berko}) in a very natural way a subset of the Berkovich polydisc 
$S_\Berko \subset \Bn_K$ which satisfies $S\subset S_\Berko$ w.r.t. the inclusion $\Kon \subset \Bn_K$. 
Then we prove our key result (theorem \ref{theodimsub}):
\[ \dim(S) = \max_{x\in S_\Berko} ( d(\H(x)/K).\]
Berkovich points are unavoidable for such a formula because when $x\in S \subset \Kon \subset \Bn_K$, then 
$\H(x)=K$ so that 
$d(\H(x) /K)=0$. \par
This formula allows us to give a new treatment of the dimension of subanalytic sets. 
For instance the formula $\dim(S \cup T) = \max(\dim(S) \cup \dim(T))$ becomes transparent. 
In addition, we manage to prove 
the invariance of dimension under subanalytic bijection, and that the dimension 
decreases under subanalytic map \ref{propinvar}, and a formula 
which relates dimension with dimension of fibers \ref{propdimfamily}.

\subsection{Dimensions with many sorts}
We prove theorem \ref{theodimmixed} with two methods. 
First in section \ref{sec2}, we define a notion of relative cell.
A relative cell $C \subset \KmGn$  is a 
family of cells (in the sense of o-minimality) of $\Gamma^n$ which 
are parametrized by a definable set of $K^m$. 
With the help of an interesting in itself cell decomposition theorem 
 \ref{theomixedcelldec} we can define the 
mixed dimension of subsets $S\subset \KmGn$ as a finite subset of
$\llbracket 0,m\rrbracket \times \llbracket 0,n \rrbracket$, and prove that it satisfies some nice properties stated in \ref{theodimmixed}. \par 
We give a second treatment of this mixed dimension in section  \ref{section3sortes}, 
which is an extension of the methods 
of section \ref{sec1.2}.
By this we mean that we manage to define the mixed dimension of 
some definable $S \subset \KmGn$ using some analogues of transcendental
degrees of some generic points of $S$.
This technique  even allows to define a dimension for definable subsets 
$S \subset K^m \times \Gamma^n \times k^p$  (where $k$ is the residue field). 
Precisely $\dim(S)$ is a finite subset of $\llbracket 0,m\rrbracket \times \llbracket 0,n \rrbracket \times \llbracket 0, p \rrbracket$.
The main idea (see definition \ref{defi:d3sortes}) is to associate some $3$-uple in $\N^3$ to some points of $S(L)$ where 
$K\to L$ is some algebraically closed extension of $K$. \par 
We want to mention some link with the notion of $VF$-dimension (denoted by $\dim_{VF}$) introduced in the work of 
Ehud Hrushovski and David Kazhdan \cite[section 3.8]{HrKa}. 
If $S \subset K^m \times \Gamma^n \times k^p$ is a definable set, 
then $S$ can also be considered as a definable subset of $K^m \times \Gamma^n \times RV^p$. 
The following compatibility holds: 
\[ \dim_{VF}(S) = \max( a \st (a,b,c) \in \dim(S) )\]
where on the left hand side, $\dim_{VF}$ is the notion introduced in 
\cite{HrKa} and on the right hand side, this is the dimension that we introduce in sections 
\ref{sec2} and \ref{section3sortes}.\par 

\subsection*{Organization of the paper}
The section \ref{preliminaries} is for preliminaries: subanalytic sets in \ref{prelisousanalytiques}, the notion of o-minimal cells in \ref{sec:Cells}, 
dimension for the $\Gamma$-sort in \ref{sec1.4}, and the dimension of semi-algebraic sets in \ref{subsection:dimACVF}. \par 

In section \ref{sec1.2}, we expose our treatment of the dimension of subanalytic sets with tools from Berkovich spaces. 
In \ref{sec:sub_ber}, we explain what we mean by a subanalytic set of the Berkovich closed unit polydisc.
In \ref{sec:gen_prop} we relate the definition of dimension of subanalytic sets as defined by L. Lipshitz and Z. Robinson 
to a new definition involving some degrees of completed residue field of Berkovich spaces. This allows to prove stability of the dimension under 
subanalytic isomorphism and a formula involving dimension of fibers. In \ref{sec:closure}, we give new proofs of results of \cite{LRdim}, 
concerning the dimension of the closure. Finally, we make a remark concerning C-minimal structures in \ref{sec:Cminimal}. \par 

In section \ref{sec2}, definable means definable in ACVF or in the analytic language of L. Lisphitz. 
We define the mixed dimension of subanalytic sets of $\KmGn$ and 
prove its basic properties: invariance under definable isomorphism 
and behavior under projection. The method used is a \emph{relative cell decomposition}. 
In \ref{sec2.1}, we explain and prove our relative cell decomposition theorem.
We use this in \ref{sec2.2} to define mixed dimension of a definable set of $\KmGn$, which is a finite subset of $\N^2$,  and prove 
its invariance under definable bijection.
In \ref{sec2.3} we prove a formula for the dimension of some projection. 
In \ref{sec:dimN}, we explain how to define the $\N$-valued dimension from the $\N^2$-valued one. \par 

In section \ref{section3sortes} definable still means definable in ACVF or the analytic language.  
We obtain a new approach to mixed dimension using degrees of generic points. 
In \ref{sec:2sortsgen}, we develop it for definable sets of $\KmGn$ and explain that we obtain the same dimension as in section \ref{sec2}.
In \ref{defi:d3sortes} we explain how to extend this dimension to a dimension for definable sets of $\KmGn \times k^p$. 
The reader interested by a quick treatment  of the mixed dimension should prefer this section.\par 

Finally in section \ref{sec3}, we explain how the results of this 
paper can be connected to some tropicalization of Berkovich spaces and allow a conceptual proof of 
theorem \ref{theoimprovduc}.

\textbf{Acknowledgements.} I would like to express my deep gratitude to Jean-Fran\c{c}ois Dat and Antoine Ducros for their constant support during this work.


\section{Preliminaries}
\label{preliminaries}
For Berkovich spaces, some introduction sources are \cite{DucBou, Berko90, NicFor}. 
We will set $\Bn_K := \mathcal{M}(\Tn)$ for the Berkovich closed unit ball of dimension $n$.
\subsection{Subanalytic sets}
\label{prelisousanalytiques}
In this part we consider an algebraically closed  non-Archimedean field $K$.
\subsubsection{Basic properties}
Subanalytic sets have been introduced by Leonard Lipshitz in 
\cite{Lip88} and \cite{Li_rig}. 
This study has then been continued by L. Lipshitz and Z. Robinson in \cite{LR_ring} and \cite{LRline,LRonedim,LRcurve,LRdim}.
We want here to present this theory as a successful way to give 
a description of the image of an analytic morphism between affinoid spaces. \par 
Semianalytic sets of an affinoid space $X$,are the Boolean combinations 
of the sets defined by inequalities 
$\{ |f| \leq |g| \}$ where $f$ and $g$ are analytic functions on $X$.
Unfortunately, in general, the image of an affinoid map is not semianalytic. This forces to have a new idea. \par 
In \cite{Li_rig} L. Lipshitz introduces a set of functions which is bigger than $\Tn$.
The precise definition has two ingredients.
The first one is to make divisions (this idea already appeared for subsets of  $\Zp^n$ in \cite{DenVDD}).
The second one is to consider analytic functions on the open ball (and not only on the closed 
ball as is classically made in rigid or Berkovich geometry). 
However, since the algebra of analytic functions on the open ball of $K$ is not nice enough (not Noetherian for 
instance), one has to use some alternative ring.\par 
More precisely, if $m,n\in \N$, one introduces a ring $\Smn$ of analytic functions on
$\Kom \times \Koon$.  
These rings, studied in detail in \cite{LR_ring} satisfy most of the good properties of 
commutative algebra satisfied by 
the Tate algebra, except Noether Normalization.
In particular, they are Noetherian (see 
\cite[p.9]{LR_ring} for a list of these good properties).\par 
In fact $\Smn$ 
depends on some choice (cf \cite[2.1.1]{LR_ring}) 
of a discrete valuation subring $E \subset K^\circ$, and should be rather denoted by $\Smn(E,K)$.
Precisely, if $\{a_i\}_{i\in \N}$ is a sequence of $K^\circ$ which tends to $0$, we 
associate to it the subring of $K^\circ$: 
\[ \Big( E[a_0,a_1,\ldots]_{ \{a \in E[a_0,a_1,\ldots] \st |a|=1 \} } \Big) \ \widehat{} \ . \]
Let $\mathfrak{B}$ denotes the set of subrings of $K^\circ$ of this form. 
\begin{defi}[{ \cite[2.1.1]{LR_ring} }] 
\[ 
\Smn (E,K) := K \otimes_{K^\circ} 
\Big( \varinjlim_{B \in \mathfrak{B}} \big( B\langle X_1,\ldots,X_m\rangle [[\rho_1,\ldots,\rho_n]] \big) \Big) .\]
\end{defi}
In many cases, a more concrete description can be given.
For instance when $K=\Cp$ and $E = \Qpunr$,
\[\Smn =\Cp \ctp_{\widehat{\Qpunr}^\circ} \Big( \big( \widehat{\Qpunr}^\circ \big) \langle X_1 \ldots X_m \rangle [[ \rho_1 \ldots \rho_n ]] \Big)  .\]
If $K= \Qp$ and $E = \Zp$, then 
\[\Smn = \Qp \otimes_{\Zp} \Zp \langle X_1,\ldots,X_m\rangle [[ \rho_1,\ldots,\rho_n]].\]

\subsubsection{The class of $D$-functions}
The $K$-algebra of $D$-functions  is defined as the smallest subalgebra of 
functions $f : \Kon \to K$ such that 

\begin{enumerate}
\item 
If  $f \in \Tn$, then 
$f: \Kon \to K$ is a $D$-function.
\item
If $f,g : \Kon \to K$ are $D$-functions, then
$D_0(f,g)$ and $D_1(f,g)$ are $D$-functions 
where  
\[\begin{array}{cccc}
D_0(f,g) : & \Kon & \to & K  \\
           & x  & \mapsto & 
\begin{cases}           
\frac{f(x)}{g(x)} & \text{if} \ |f(x)| \leq |g(x)| \neq 0 \\
        0 & \text{otherwise}   
        \end{cases}           
\end{array}
\]

\[\begin{array}{cccc}
D_1(f,g) : & \Kon & \to & K \\
           & x  & \mapsto & 
\begin{cases}           
 \frac{f(x)}{g(x)} & \text{if} \ |f(x)| < |g(x)|  \\
        0 & \text{otherwise}   
        \end{cases}           
\end{array}
\]

\item Let 
$f_1,\ldots, f_m,g_1,\ldots g_n$ be some $D$-functions $\Kov{p} \to K$. 
Let us assume that for all $x\in \Kov{p}$, and that for all $i,j$ 
$|f_i(x)| \leq 1$ and $|g_j(x)| <1$. 
If $F \in \Smn$ then 
\[ 
\begin{array}{cccc}
F(f_1,\ldots,f_m,g_1,\ldots,g_n) : & \Kov{p} & \to & K \\
                  & x & \mapsto & F(f_1(x), \ldots, f_m(x), g_1(x) , \ldots, g_n(x) ) 
                  \end{array}\]
is a
$D$-function.
\end{enumerate}

\begin{defi}
A set $S \subset \Kon$ is subanalytic if it is a finite boolean combination of sets
$\{x\in \Kon \st |f(x)| \leq |g(x)| \}$ where $f,g$ are $D$-functions.
\end{defi}
Note that if $f_1,\ldots, f_N \in \Tn$, then the associated $K$-affinoid space 
$X= \{x\in \Kon \st |f_i(x)|=0\}$ is a subanalytic set because the 
$f_i$'s are $D$-functions, and 
$f(x)=0$ is equivalent to 
$|f(x)| \leq 0$. 
From a geometric view point the interest of subanalytic sets is: 
\begin{theor}
 \cite[th. 3.8.2]{Li_rig} 
 Let $X \subset \Kon$, $Y\subset \Kom$ be some affinoid spaces, and 
 let $f : X\to Y$ be some affinoid map. 
 Then $f(X) \subset Y \subset \Kom$ is a subanalytic set of $\Kom$.
\end{theor}
In fact, the statement of 
\cite[th. 3.8.2]{Li_rig} is more general, and is expressed in a model theoretical way that we present now.
\begin{defi}
A formula in the language $\Ldan$, or $\Ldan$-formula 
is a first order formula  $\varphi(x)$ in the variables $x=(x_1,\ldots,x_n)$, written with symbols 
$f$ for $f$ a $D$-function, inequalities $|f|\leq |g|$, $|f|< |g|$ for $f,g$ some $D$-functions, and logical symbols 
$ \forall, \exists, \neg$. 
\end{defi}

\begin{rem}
 It follows from the above definitions that a subset $S \subset \Kon$ is subanalytic if and only if there exists 
 some quantifier free $\Ldan$-formula $\varphi(x)$ such that 
 $S = \{x\in \Kon \st \varphi(x) \ \text{is true} \}$.
\end{rem}

Let $K \to L$ be a non-Archimedean algebraically closed field extension, $\phi$ some $\Ldan$-formula in the variables $(x_1,\ldots,x_n)$.
Let $x\in (L^\circ)^n$, then one can 
evaluate the truthfulness of $\phi(x)$, because a $D$-function $f$ in n variables defined over $K$ induces also a function 
$(L^\circ)^n \to L$.
\begin{theo}
\label{theoquantel}[{\cite[3.8.2]{Li_rig} } and  
{\cite[4.2]{LRMod}} for uniformity]
Let $\varphi(x)$ be a $\Ldan$-formula in the variables $(x_1,\ldots,x_n)$. 
There exists a quantifier free formula $\psi(x)$ such that for all non-Archimedean extension of algebraically closed valued fields 
$K \to L$, for all $x\in (L^\circ)^n$, 
$\phi(x) \Leftrightarrow \psi(x)$. In other words 
\[ \{x\in (L^\circ)^n \st \phi(x) \ \text{is true} \} = \{x\in (L^\circ)^n \st \psi(x) \ \text{is true} \}.\]
\end{theo}
This result is another way to say that the projection of a subanalytic set is also subanalytic. 
This result implies the stability of subanalytic sets under affinoid maps mentioned above.
\begin{rem}
\label{remuniform}
If $S\subset \Kom$ is subanalytic, and defined by the formula $\varphi$, and if  
$K \to L$ is a non-Archimedean extension of algebraically closed fields, we will set   
\begin{equation}
 S(L):= \{x\in (L^\circ)^n \st \phi(x) \ \text{is true} \} .
\end{equation}
It is crucial to understand that this does not depend on $\varphi$ but only on $S$ according to the above theorem.
\end{rem}
We mention the following result:
\begin{lemme}
\label{propcont}
Let $f : \Kon \to K$ be a $D$-function. 
There exists a decomposition 
$\Kon = X_1\cup \ldots \cup X_N$ in subanalytic sets such that for all $i$, 
$f_{|X_i} : X_i \to K$ is continuous.
\end{lemme}
Tis can be proved easily by induction on the definition of the $D$-function $f$.
\subsubsection{Quantifier elimination with three sorts: Algebraic setting}
\label{subsubsec:ACVF}
Let $\mathcal{L}_1$ be the language with two sorts $K$, $\Gamma$, the language of 
rings for $K$, i.e. symbols $+,\cdot,-,0,1$ interpreted in $K$, a symbol 
$|\cdot | : K \to \Gamma_0$, and symbols 
$\leq, <, \cdot,0$ interpreted on $\Gamma_0$.  
Then the theory ACVF of algebraically closed valued fields with two sorts $K$ and $\Gamma$  
interpreted in $\mathcal{L}_1$ has quantifier elimination.
\par 
We now consider the language $\mathcal{L}_2$ with three sorts $K, \Gamma$ and $k$. 
$\mathcal{L}_2$  contains the symbols of $\L_1$, and 
we add to it the ring language 
$+,-,\cdot,0,1$ for $k$, a function 
$\tilde{\cdot} : K \to k$ such that the restriction of 
$\tilde{\cdot}$ to $K^\circ$ is the classical reduction  (by this we mean that it is
a surjective ring morphism whose kernel is $K^{\circ \circ}$) and 
$\tilde{\cdot}$ is extended outside $\Ko$ by 
 $\tilde{x} =0$ if $|x|>1$. 
Finally, there is a function 
$Res : K^2 \to k$ such that 
\[
\begin{array}{cccc}
\text{Res}: & K^2 & \to & k \\
            & (x,y) & \mapsto & 
\begin{cases}
 \widetilde{\left(\frac{x}{y}\right)} \ \text{if} \ |x| \leq |y| \neq 0 \\
0  \ \text{otherwise}
\end{cases} 
\end{array}\]
Then the theory ACVF with the three sorts $K,\Gamma,k$ has quantifier elimination in the language $\mathcal{L}_2$.
See \cite[theorem 2.1.1]{HHM}) for some explanations.
\par 
We want to insist on the fact that if one removes $Res$ from the language $\mathcal{L}_2$, 
the quantifier elimination does not hold anymore.
Indeed, the set 
\[ S = \{ (x,y,\lambda) \in K^2 \times k \st \text{Res}(x,y)= \lambda \} \] 
which is nothing else than the graph of $Res$,
is definable by a first order formula using $\tilde{\cdot}$ but with a quantifier (to express $\frac{x}{y}$). 
But $S$ can not be definable without quantifier if we do not allow the use of $Res$.

\subsubsection{Quantifier elimination with three sorts: Analytic setting}
\label{subsubsec:an}
All these quantifier elimination results have their counterpart in the analytic language. 
Let $K$ be some algebraically closed non-Archimedean field.  \par 
Let $\mathcal{L}_1^D$ be the language with two sorts $K$, $\Gamma$, the language of 
rings for $K$, i.e. symbols $+,\cdot,-,0,1$ interpreted on $K$, a symbol $f$ for 
each $D$-function $f$, a symbol 
$|\cdot | : K \to \Gamma_0$, and symbols 
$\leq, <, \cdot$ interpreted on $\Gamma_0$.  
In this setting, quantifier elimination takes the following form: 
\begin{theor} \cite[th. 3.8.2]{Li_rig}
For any formula $\varphi(x,\gamma)$ in the language $\mathcal{L}_1^D$, there exists a quantifier free formula $\psi$ such that for 
any algebraically closed non-Archimedean extension $K \to L$, 
\[  \{ (x,\gamma) \in (L^\circ)^n \times \Gamma_L ^m \st \varphi(x,\gamma) \} = 
 \{ (x,\gamma) \in (L^\circ)^n \times \Gamma_L ^m \st \psi(x,\gamma) \}. 
\]
\end{theor}

\par 
We now consider the language $\mathcal{L}_2^D$ with three sorts $K, \Gamma$ and $k$, with the symbols of $\mathcal{L}_1^D$, and 
we add the ring language 
$+,-,\cdot,0,1$ for $k$, a function 
$\tilde{} : K \to k$ such that the restriction of 
$\tilde{}$ to $K^\circ$ is the classical reduction, i.e. is a surjective ring morphism whose kernel is $K^{\circ \circ}$ 
and $\tilde{x} =0$ if $|x|>1$. Finally, there is a function 
$Res : K^2 \to k$ such that 
\[
\begin{array}{cccc}
\text{Res}: & K^2 & \to & k \\
            & (x,y) & \mapsto & 
\begin{cases}
 \widetilde{\left(\frac{x}{y}\right)} \ \text{if} \ |x| \leq |y| \neq 0 \\
0  \ \text{otherwise}
\end{cases} 
\end{array}\]
Then, the proof of \cite[th. 3.8.2]{Li_rig} can be extended to:
\begin{theor} 
For any formula $\varphi(x,\gamma, \lambda)$ in the language $\mathcal{L}_2^D$, there exists a quantifier free formula $\psi$ such that for 
any algebraically closed non-Archimedean extension $K \to L$, 
\[  \{ (x,\gamma) \in (L^\circ)^n \times \Gamma_L ^m\times k_L^p \st \varphi(x,\gamma) \} = 
 \{ (x,\gamma) \in (L^\circ)^n \times \Gamma_L ^m \times k_L^p \st \psi(x,\gamma) \}. 
\]
\end{theor}
\begin{rem}
 \label{rem:ens_fini}
 One consequence of quantifier elimination is that if $S$ is a subanalytic set defined over some non-Archimedean algebraically closed field $K$ 
 which is finite, then for any non-Archimedean algebraically closed extension $K \to L$, 
 $S = S(L)$. 
 Indeed, if $S = \{x_1,\ldots,x_N\}$, then $X$ can be defined by the formula 
 \[ S = \{ x \in K^n \st \bigvee_{i=1 \ldots N} x= x_i \}. \]
 And with this definition, it is clear that $X = X(L)$.
\end{rem}
 We finally gather some results  that we will need.
 
 \begin{lemme}
\label{finite_image}
Let $X\subset \Gamma^n \times k^p$ be a definable set and let 
$f : X \to K^m$ be a definable map.
Then $f(X)$ is finite.
\end{lemme}

\begin{proof}
Since there are only countably many coefficients involved in the definition of $X$ and $f$, reducing the field $K$ if necessary, we can assume that $k$ and $\Gamma$ are countable.
Hence, $X$ is countable. But if $f(X)$ was infinite, it would contain some nonempty open ball, and would be uncountable (K being complete is uncountable). 
\end{proof}
 
\begin{prop}
\label{propfonccont}
Let 
$X\subset \Kom \times \Gamma^n$ be a subanalytic set, and 
$f: X \to \Gamma$ a subanalytic map. 
Then there exists a decomposition 
$X=X_1 \cup \ldots \cup X_N$ in subanalytic sets such that 
for all $i$ 
$f_{|X_i} : X_i \to \Gamma$ is continuous. \par
Actually, it is even defined as 
\[f_{|X_i} = \sqrt[a_i]{|g_i(x)| \alpha^{u_i}} \]
where $g_i$ is a continuous subanalytic map 
$U_i \to K$ for some subanalytic $U_i \subset \Kom$, and 
$u_i \in \Z^n$, $a_i\in \N^*$. 
\end{prop}
The analogous statement for ACVF is also true (and can be done with the same proof).
\begin{proof}
Let 
$G:= \Graph(f) \subset \Kom \times \Gamma^n \times \Gamma$ which is subanalytic.
We will use the variable 
$x$ (resp. $\alpha$, $\gamma$) for the variable of $\Kom$ (resp. $\Gamma^n$, $\Gamma$). 
It is enough to consider 
the case where 
$G$ is a finite intersection of subsets defined by inequalities 
\[ |F(x)| \gamma^a \bowtie |G(x)| \alpha^u \]
where $a\in \N$, $u\in \Z^n$, $\bowtie \in \{ =,>,<\}$ and 
$F,G : \Kom \to K$ are $D$-functions.
We can then restrict to  subanalytic sets where 
$F(x)$ and $G(x) \neq 0$, hence their quotient 
$\frac{F}{G}$ will be $D$-functions 
$\Kom \to K$. 
Hence $G$ appears as a finite union of subsets of the form 
\[ I_j = \{ (x,\alpha, \gamma )\in \Kom \times \Gamma^n \times \Gamma 
\st |g_j(x)| \alpha^{u_j}< \gamma^{b_j} < |h_j(x)| \alpha^{v_j} \} \ j=1\ldots M\]
and 
\[ \{ (x,\alpha, \gamma )\in \Kom \times \Gamma^n \times \Gamma 
\st |f_i(x)|  \alpha^{u_i} = \gamma^{a_i} \} \ i=1\ldots N \]
plus some additional conditions involving only the variables $(x,\alpha)$.\\
Since by assumption, for all 
$(x,\alpha) \in X$, 
$G_{(x,\alpha)} = \{ \gamma \in \Gamma \st (x,\alpha,\gamma) \in G\}$ is a singleton, 
we can in fact remove the intervals $I_j$. 
We then set temporarily 
\[X_i = \{(x,\alpha) \in X \st f(x,\alpha) = 
\sqrt[a_i]{   |f_i(x)|  \alpha^{u_i} } \}. \]
We then have 
\[f_{|X_i}(x,\alpha) = \sqrt[a_i]{ |f_i(x)| \alpha^{u_i} } \] 
with 
$f_i : \Kon \to K$ an $D$-function.
Now if we apply lemma
\ref{propcont}
to the map 
$f_i : \Kon \to K$ we know that 
we can shrink the $X_i's$ so that the map 
\[
\begin{array}{ccc}
X_i & \to & K\\
(x,\alpha) & \mapsto & f_i(x) 
\end{array}
\]
is continuous. 
As a consequence,  
\[
\begin{array}{ccc}
X_i & \to & \Gamma \\
(x,\alpha) & \mapsto & |f_i(x)| 
\end{array}
\]
is continuous, and then 
\[
\begin{array}{cccc}
f_{|X_i} : & X_i & \to & \Gamma \\
& (x,\alpha) & \mapsto & \sqrt[a_i]{ |f_i(x)| \alpha^{u_i} } 
\end{array}
\]
is also continuous.
\end{proof}
\subsection{Cells}
\label{sec:Cells}
Let us now explain the method used to define the mixed dimension with value 
in finite lower subsets of $\N^2$ and how we prove theorem \ref{theodimmixed}. 
First one has to keep in mind that for definable subsets $S \subset \Gamma^n$, 
one way to define and study the dimension of $S$, is to use 
the cell decomposition theorem \cite[Chapter 3]{VDD}. 
What does this theorem say?\par 
If $(i_1,\ldots,i_n) \in \{0,1\}^n$ one defines inductively what it means for  
$C \subset \Gamma^n$ to be an $\icell$-cell. 
\begin{itemize}
 \item A $(0)$-cell  (resp. a $(1)$-cell) is a singleton 
(resp. an open interval) in $\Gamma$.
\item If $C\subset \Gamma^n$ is a $\icell$-cell, 
an $(i_1,\ldots,i_n,0)$-cell (resp. $(i_1,\ldots,i_n,1)$-cell) is a subset 
$\{(x,\gamma) \in C \times \Gamma \st \gamma=f(x)\}$ 
(resp. $\{(x,\gamma) \in C \times \Gamma \st f(x)<\gamma<g(x)\}$) 
where $f:C \to \Gamma$ is a continuous definable function (resp. $f,g: C \to \Gamma$ are continuous definable functions such that 
$f<g$). 
\end{itemize}

Roughly speaking, an $\icell$-cell is build 
from graphs (for the $0$'s appearing in $i_1,\ldots,i_n$)  
and open domains delimited by functions $f<g$ (for the $1$'s which appear in $i_1,\ldots,i_n$). The dimension of an 
$\icell$-cell is $i_1+\ldots +i_n$.
The cell decomposition asserts that any definable set $S\subset \Gamma^n$ can be partitioned into 
finitely many cells $\{C_j\}_{j=1 \ldots N}$, and the dimension of $S$ can be defined as the maximum of the dimension 
of the $S_j$'s. Alternatively, 
\begin{equation}
\label{dimcell}
 \dim(S) = \max_{\underset{ C \subset S}{C \ \text{an} \ \icell\text{-cell} } } i_1 + \ldots + i_n.
 \end{equation}
In our situation, 
if $S \subset \KmGn$ is definable, for each $x\in K^m$, the fibre 
of $S$ above $x$ is a definable set $S_x \subset \Gamma^n$, hence can be decomposed into 
finitely many cells. The idea is to prove that the cell decomposition behaves 
nicely when $x$ moves. 
To formalize this we   
 introduce the notion of a cell $C \subset \KmGn$. 
So to speak they are \emph{relative cells}. 
If $X \subset K^m$ is a definable subset, we define what it 
means for $C \subset X \times \Gamma^n$ 
to be a $X-(i_1,\ldots,i_n)$ cell (definition \ref{deficell}), 
which is a continuous definable family of 
$(i_1,\ldots,i_n)$-cell (in the above sense) parametrized by $X$. 
Then we do things in close analogy
with what already exists for cells of $\Gamma^n$: we prove 
a cell decomposition theorem (theorem \ref{theomixedcelldec}) 
for definable subsets of $\KmGn$, then we define the dimension 
of a $X-\icell$-cell to be $(\dim(X), i_1+ \ldots+ i_n)\in \N^2$, and eventually, 
if $S\subset \KmGn$ is definable, we define 
\[ \dim(S) = \{ \dim(C) \st \ C \subset S \ \text{is a cell} \}. \]
Hence by definition, $\dim (S)$ is a finite lower set of $\N^2$ included in 
$\llbracket O,m\rrbracket \times \llbracket 0,n \rrbracket$. 
This dimension theory satisfies the properties mentioned in theorem \ref{theodimmixed}.
Once we have proved the (mixed) cell decomposition theorem 
(theorem \ref{theomixedcelldec}) the last non-trivial part of 
theorem \ref{theodimmixed} is $(iii)$. 
To handle it, we reduce to 
the simple situation of lemma 
\ref{lemmedimproj}, where we use the simple behavior of definable maps 
$f:K \to \Gamma$.\par
As will be seen in sections \ref{sec1.2}, the dimension theories 
of definable subsets $X \subset K^n$ in ACVF and 
subanalytic sets $X \subset \Kon$ share many properties. 
As a consequence, 
the results exposed in section \ref{sec2}, especially concerning the mixed cell decomposition, 
work in the frameworks of ACVF and subanalytic sets without difference. 
In \ref{sec1.4} we briefly expose the dimension 
theory of definable sets $S \subset \Gamma^n$ when 
$\Gamma$ is a totally ordered abelian divisible group, referring to \cite{VDD}\footnote{The case of 
totally ordered abelian divisible groups is only one case (and maybe the simplest) of an  o-minimal theory, 
which is exposed in the excellent book \cite{VDD}.}. 
In section \ref{sec2} the proofs are general enough to apply both 
to definable sets in ACVF and to subanalytic sets. Hence in this part, \emph{definable set} 
will mean definable in ACVF or subanalytic. 
In \ref{sec2.1}, we prove the mixed cell decomposition theorem 
for definable sets 
$S \subset \KmGn$. 
Sections   
\ref{sec2.2} and \ref{sec2.3} are devoted to the proof of 
the general properties of the mixed dimension of definable
subsets $S \subset \KmGn$ that we have announced.\\ 
In \ref{sec3} we explain how this relates to previous works. 

\subsection*{A word of warning} 
\label{subsec:warning}
In the proofs which will follow, we will always work with variables $\gamma \in \Gamma$. 
However, 
it may happen that for some $x\in K^m$, $|f(x)|= 0$. 
For instance the 
formula $|0| = \gamma$ does not have any sense if we want to quantify on a 
variable $\gamma \in \Gamma$. 
A correct formulation is to work with
$\Gamma_0 = \Gamma \cup \{0\}$. \par 
For instance, in all the constructions which will follow, at some point of the 
proofs, we should have distinguished between the cases 
$f(x)=0$ and $f(x)\neq 0$ when we consider (definable) functions $f : K^m \to K$. 
However, these two sets are definable, and on the first one, 
$|f(x)|=0$ is constant. 
On the other one, $f(x)\neq 0$, hence $|f(x)| \in \Gamma$.
Another way to see it is that  
\[(\Gamma_0)^n = \coprod_{k=0}^n ( \coprod_{ \binom{n}{k} } \Gamma^k )) .\] 
So that we can always replace $\Gamma_0$ by $\Gamma$. \par 
We have chosen to systematically omit this point in the proofs, and to 
write $\Gamma$ instead of $\Gamma_0$ in order to avoid some unprofitable complications. 
The careful reader should have this in mind.

\subsection{Dimension for the $\Gamma$-sort}
\label{sec1.4}

If $(\Gamma, <)$ is a totally ordered abelian group, the intervals 
$]a,b[ := \{ \gamma \in \Gamma \st a<\gamma<b \}$ form a basis of neighborhood 
for some topology, which is called the \emph{interval topology}. We will always refer to this topology on $\Gamma$. \par 
Remind that if $K$ is an algebraically closed non-trivially valued field, 
$\Gamma : = |K^*|$ is a totally ordered divisible commutative group.\par
A definable subset of $\Gamma^n$ is then a boolean combination of subsets 
\[ \{ \gamma \in \Gamma^n \st 
\gamma^a \lambda \}\]
where $\lambda\in \Gamma$  and $u\in \Z^n$.

\begin{defi}
\label{deficellGamma}
Let $n\in \N$, 
$(i_1,\ldots, i_n) \in \{0,1\}^n$.
An \textbf{$\icell$-cell} is a definable subset $C\subset \Gamma^n$ that we define inductively: 
\begin{enumerate}[i)]
\item 
A $(0)$-cell is a singleton $\{ \gamma\}$ with $\gamma \in \Gamma$. \\
A $(1)$-cell is in interval $]a,b[$ of $\Gamma$.
\item 
If $C \subset \Gamma^n$ is a $\icell$-cell, and 
$f,g : C \to \Gamma$ are continuous definable maps, and assume that for all $x\in C$, 
$f(x)<g(x)$. Then \\
$\Graph(f) :=\{ (x,y) \in (\Gamma)^n \times \Gamma \st x\in  C \ \and \ 
y=f(x) \}$ is an 
$(i_1,\ldots, i_n,0)$-cell. \\
$]g,f[ := \{(x,y) \in (\Gamma)^n \times \Gamma \st x\in  C \ \and \ 
f(x) < y <g(x) \}$, as well as $]- \infty , f[$ and $]f, +\infty[$ are 
$(i_1,\ldots, i_n,1)$-cells.

\end{enumerate}
A cell is an $\icell$-cell for some 
$\icell$, and a cell of dimension $d$ is an 
$\icell$-cell with 
$i_1+ \ldots +i_n =d$.
\end{defi}

\begin{figure}[h]
\begin{subfigure}[b]{0.45\linewidth}

\begin{tikzpicture}[scale=0.5]
\draw (-4,-2) -- (4,-2);
\draw (-4,-2)--(-4,4);
\draw [<->,thick] (-3,-2.2) --(3,-2.2) node[below,midway]{C};
\draw [dashed] (-3,{sin(-3 r)}) -- (-3,-2);
\draw [dashed] (3,{sin(3 r)}) -- (3,-2);
\draw (0,-0.5) node{$f$};
\draw plot [domain=-3:3] (\x,{sin(\x r)});
\end{tikzpicture}
\caption{$\Graph(f)$, an $(i_1,\ldots,i_n,0)$-cell}
\end{subfigure}
~~
\begin{subfigure}[b]{0.45\linewidth}

\begin{tikzpicture}[scale=0.5]
\draw (-4,-2) -- (4,-2);
\draw (-4,-2)--(-4,4);
\draw [<->,thick] (-3,-2.2) --(3,-2.2) node[below,midway]{C};
\draw [dashed] (-3,{2+exp(-3+3)} )-- (-3,-2);
\draw [dashed] (3,{2+exp(-3-3)} )-- (3,-2);
\draw plot [domain=-3:3] (\x,{2+exp(-3-\x)});
\draw (0,-0.5) node{$f$};
\draw (0,2.5) node{$g$};
\draw plot [domain=-3:3] (\x,{sin(\x r)});
\fill [pattern = north east lines]
 plot [domain=-3:3] (\x,{2+exp(-3-\x)} )
--
 plot [domain=3:-3] (\x,{sin(\x r)}) 
-- cycle;
\end{tikzpicture}
\caption{$]f,g[$, an $(i_1,\ldots,i_n,1)$-cell}
\end{subfigure}
\caption[Picture of cells]{We suppose that 
$C\subset \Gamma^n$ is an $(i_1,\ldots,i_n)$-cell and 
that $f,g : C \to \Gamma$ are continuous definable functions such that 
$f<g$. The horizontal axis represents $C$, and the vertical axis represent $\Gamma$.
On the left the graph of  we obtain $\Graph(f)$.
On the right, the hatched area represents $]f,g[$.}  
\end{figure}
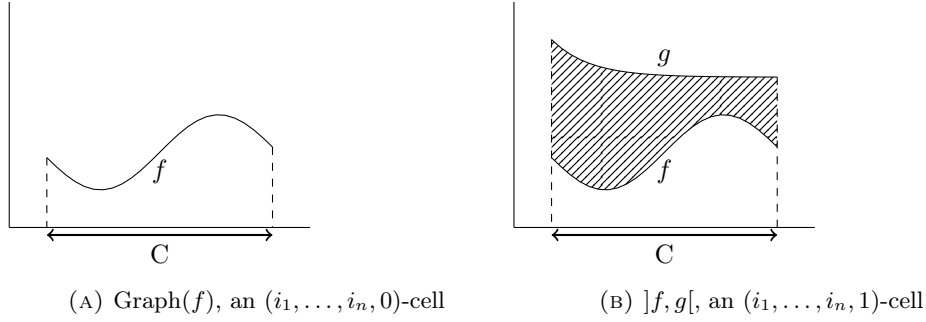

The following result can be found in 
\cite[3.2.11]{VDD}
\begin{theo}[Cell decomposition theorem]
Let $X$ be a definable subset of $(\Gamma)^n$. There exists a partition 
\[X = \coprod_{i=1}^m C_i\] where each $C_i$ is a cell.
\end{theo}

\begin{defi}
Let $X$ be a definable subset of $(\Gamma)^n$. The following numbers are equal, and will be denoted by 
$\dim(X)$.
\begin{enumerate}
\item 
\[\max_{\underset{C \ \text{is a cell} }{C \subset X} } (\dim (C) )  \]
\item 
\[\max_{k=1\ldots M} (\dim(C_k)) \]
where 
$X = \coprod_k C_k$ is any decomposition of $X$ into cells. This number does not depend on the decomposition.
\item 
The greatest integer $d$ such that there exists a coordinate projection 
$\pi : (\Gamma)^n \to (\Gamma)^d$ for which $\pi(X)$ has non empty interior.
\end{enumerate}
\end{defi}

These equivalences are proven in \cite[part 4]{VDD}, as well as this proposition:
\begin{prop}\cite[4.1.6]{VDD}
\label{propimagedim}
Let $f:S \to T$ be a definable maps between definable sets of some $\Gamma^n$, then 
$\dim(f(S)) \leq \dim(S)$.
\end{prop}  

\subsection{Dimension of definable subsets $S \subset K^n$ in ACVF}
\label{subsection:dimACVF}
If $K$ be a model of ACVF.
One can associate to definable sets $S \subset K^n$ a dimension, 
denoted $\dim(X)$ which satisfies many natural properties. 
This is due to Lou Van Den Dries \cite{VDDdim}, where it is proved for more general Henselian fields. 
A concise explanation of these results can also be found in the author's thesis \cite[section 3.2]{MarThese}. \par 
Let  $\X$ be a $K$-scheme of finite type and let $U \subset \X(K)$. 
Since $K$ is algebraically closed, we can identify $U$ with a subset of $\X$ and we denote by $\dimZ(U)$ the Zariski dimension of the Zariski closure of $U$ in $\X$ : 
\[\dimZ(U) := \dim ( \overline{U}^{\text{Zar}}).\]
\begin{defi}
Let $\X$ be an affine $K$-scheme of finite type. 
We will say that 
$U \subset \X(K)$ is a special open subset if it is a finite intersection: 
\[U = \Big( \bigcap_{i=1}^M \{ x\in \X(K) \st |f_i(x)| < |F_i(x)| \} \Big) \cap 
\Big( \bigcap_{j=1}^N \{ x\in \X(K) \st |g_j(x)| \leq |G_j(x)| \neq 0 \} \Big) \]
where $f_i,F_i,g_j,G_j \in \mathcal{O}(\X)$.
\end{defi}
\begin{theo}
\label{theodimK}
Let $d$ be an integer, and $S$ be a nonempty definable subset of $K^n$. 
We denote by $\dim(X)$ the integer $d$ that can be defined by the equivalent properties: 
\begin{enumerate}
\item 
There exists a decomposition 
$S = \cup_{i=1}^m S_i$ where for each $i$, $S_i \neq \emptyset$, $S_i = \X_i \cap U_i$ where 
$\X_i$ is an irreducible variety of (Zariski) dimension 
$d_i$ and $U_i$ is a special open subset of $K^n$, such that $d =\max_i d_i$. This 
number is independent of the choice of the decomposition. 
\item 
$d= \dim_{Zar} ( \overline{S}^{Zar} )$.
\item 
\noindent \[ d= \max_{\underset{K\subset L}{ x\in S(L)}} \  \td(K(x)/K)\] where 
$K \subset L$ describes the  class of small extensions of algebraically closed valued fields of $K$.
\item 
\noindent \[ d= \max_{\underset{K\subset L}{ x\in S(L)}} \  \d(K(x)/K)\] where 
$K \subset L$ describes the  class of small extensions of algebraically closed valued fields of $K$.
\item 
The greatest integer $d$ such that 
there exists a coordinate projection 
$\pi :K^n \to K^d$ such that $\pi(S)$ has nonempty interior (w.r.t. the topology on 
$K^n$ induced by the valuation).
\end{enumerate}
\end{theo}
The characterization $(4)$ does not appear in \cite{VDDdim} but can be proved using the ideas of the next section. 
Let us sketch its proof.
In one hand, it is clear that the integer $d$ defined by $(4)$ is greater or equal than 
the integer $d$ defined by condition $(3)$. 
So, in order to prove the reverse inequality using $(5)$, 
it is enough to show that if $S$ contains an open ball then there is some 
$x\in S(L)$ such that $\d(K(x)/K) =n$. Indeed, if 
$S$ contains the closed ball of center $0$ and radius $\gamma \in \Gamma$, then the norm 
$\eta : K[T_1,\ldots, T_n] \to \Gamma$ defined by 
$\eta (\sum a_\nu T^\nu) = \max |a_\nu| \varepsilon^{|\nu|}$ induces a norm on 
$L=K(T_1,\ldots,T_n)$, and we can then take $x = (T_1,\ldots,T_n)$ which belongs to 
$S(L)$, and satisfies 
$d(K(x) /K) =n$. \par 
Finally, here are some properties satisfied by this dimension.
\begin{prop}
Let $f:X \to Y$ be a morphism of definable sets. 
\begin{enumerate}
\item 
$\dim (f(X)) \leq \dim(X)$.
\item If $f$ is bijective $\dim(X)=\dim(Y)$.
\end{enumerate}
\end{prop}
\section{Dimension of subanalytic sets}
\label{sec1.2}
Let $K$ be an algebraically closed non-Archimedean field.
In this section we make a systematic study of the dimension of subanalytic sets 
$S\subset \Kon$. Using Berkovich spaces, we prove some results which were already known in \cite{LRdim}, 
but we also give new results such as invariance under subanalytic isomorphism and a formula \ref{propdimfamily} 
which calculates the dimension of a subanalytic set in terms of the dimension of some fibers $S_x$. \par 
In this part, we expose and extend the theory of dimension of subanalytic sets of 
$\Kon$ as introduced in \cite{LRdim}. 
If $S\subset \Kon$ is a subanalytic set, its dimension
$\dim(S)$ 
is defined \cite[2.1]{LRdim} as the 
greatest integer $d$ such that there exists a coordinate projection 
$\pi : \Kon \to \Kod$ for which $\pi(S)$ has nonempty interior. \par  
Many of the results we present were already known, however some were 
not (at least not to our knowledge). 
For instance, we explain how the dimension of a subanalytic set $S \subset \Kon$ has 
a nice interpretation in terms of $S_{\Berko}$, a subset of $\B_K^n$ that we attach to $S$. 
This new characterization (see theorem \ref{theodimsub} (5) ) 
enables us to give a new treatment of the dimension of subanalytic sets, 
and to prove some results which were 
not known in positive characteristic, in particular a good behavior of 
dimension in terms of fibers (see \ref{propdimfamily}), and the invariance under definable bijection (see \ref{propinvar}).\par 
Y. Firat Celikler has also published some results about dimension 
of $D$-semianalytic sets \cite{Cel05,Cel10}. 
His results apply to a broader class of sets, 
namely subsets $X \subset (K^\circ)^n$ where $K$ is 
not assumed to be algebraically closed (and 
not even a rank one valuation in \cite{Cel10}). 
Whereas our major tool is the use of 
Abhyankar points, Berkovich spaces are not mentioned in this work, 
which rate relies on generalized ring of fractions and a very interesting 
normalization lemma 
\cite[section 5]{Cel05} and \cite[section 3]{Cel10}. 
These articles contain especially results about dimension of subsets of $(K^\circ)^n$ in terms 
of the dimension of some fibers \cite[theorem 6.6]{Cel05} and\cite[theorem 5.1]{Cel10}. 
However, in these results, the characteristic is assumed to be $0$. 
 \par 
 In the same way, if the characteristic of $K$ is positive, the analytic structure on $K$ is b-minimal \cite[theorem 6.3.7]{CluLip}. 
 As a consequence, if $\cha (K)=0$, most of the results of this section follow form 
 the good dimension theory 
 which has been developed  for $b$-minimal theories (see \cite{CluIntrob} for an introduction and 
 \cite[section 4]{CluLoBmin} for a complete exposition). \par 
Let us finally mention that the first order theory defined by the analytic language $\Ldan$ is a $C$-minimal theory \cite{LRonedim}. 
$C$-minimal theories have been studied in \cite{MaStVa,HaMaCell,CubThese}, and some 
results of the dimension theory of subanalytic sets follow from the general treatment 
which is made in \cite{HaMaCell}.

\begin{rem}
For a subset $S \subset \Kon$, we will define the dimension of $S$ as the minimum 
of the integers $d$ such that there exists some coordinate 
projection $\pi : \Kon \to \Kod$ such that 
$\pi(S)$ has nonempty interior. 
We want to stress out that it is not obvious that this gives a satisfactory theory of dimension. 
For instance, let us assume that 
\begin{equation}
\label{eq:bigfield}
 \Card (\widetilde{K}) = \Card( \Ko) .
 \end{equation}
For example $K = \widehat{ \C((t))^{alg}}$ satisfies this condition. 
Let then 
$\alpha : \widetilde{K} \to \Ko$ be a bijection (a surjection would work actually), 
and let $\beta : \widetilde{K} \to \Ko$ be a section of the reduction map. 
Let then 
\[ V := \{  ( \alpha(\lambda), \beta (\lambda) ) \st \lambda \in \widetilde{K} \}.\]
Since by construction the family $\{  \beta( \lambda) \}_{\lambda \in \widetilde{K} }$ is discrete, $V$ is an infinite discrete union of singletons, hence a 
$K$-analytic submanifold
\footnote{The 
word manifold has to be understood in the \emph{naive} sense, i.e. by analogy with the definition of a real or complex analytic manifold, see \cite{SerreLG}.}
of ${\big(\Ko}\big)^2$ of pure dimension $0$. 
However, if $\pi : \big(\Ko \big)^2 \to \Ko$ is the first projection, 
$\pi(V) = \Ko$. 
In this example, the projection of a submanifold of 
dimension $0$ (as a manifold) has dimension $1$ (as a manifold).\par 
With the same kind of idea, it is also possible to find 
$S\subset \big(\Ko \big)^2$ a submanifold 
of dimension $0$ as a manifold such that the analytic dimension of 
$\overline{S}$ is $1$. 
Take 
\[\{ S = \big(x, \pi^n \beta (\alpha^{-1} (x) ) \big) \}_{x\in \Ko , n\in \N}\] 
Then $S$  
is a submanifold of dimension $0$, and contains $ \Ko \times \{0\}$ in 
its adherence. \par 
As we will see, these kind of pathologies do not appear with subanalytic sets, and as 
a matter of fact, $S$ and $V$ are not subanalytic.
\end{rem}

\subsection{Subsets of $\Kon$ and subsets of $\B^n_K$} 
\label{sec:sub_ber}

Clearly, if $K \to L$ is a non-Archimedean extension, and $\varphi$ is an 
$\Ldan$-formula defined over $K$, then $\varphi$ is also an $\Ldan$-formula defined over $L$.
We will repeatedly use the following remark.
\begin{rem}
\label{remext}
Let $K \hookrightarrow L$ be a non-Archimedean extension.
There is a continuous map 
\begin{equation}
\label{eq:defip}p: L^n \to (\An_K)^\an.
\end{equation}
If $x=(x_1,\ldots,x_n) \in L^n$, then 
for $f\in K [X]$, $|f(p(x))| = |f(x)|_L$.
\end{rem}

\begin{defi}
\label{defi:sub_berko}
 Let $S \subset \Kon$ be a subanalytic set. 
 Let $x\in \Bn_K$ and let $x_i \in \H(x)$ be the evaluation of 
 the variable $T_i$'s in $\H(x)$. 
 We define the subset 
 $S_\Berko \subset \Bn_K$ by saying that 
 $x \in S_\Berko$ if and only if 
 $(x_1,\ldots,x_n) \in S(\widehat{ \H(x)^\alg})$.
\end{defi}

We then have the following tautological result: 
\begin{lemme}
\label{lemmetaut}
Let $S$ be a subanalytic set, $K \to L$ some non-Archimedean algebraically closed extension. We consider 
$p : L^n \to \Bn_K$ as defined in remark \ref{remext}. 
Then 
\[S(L) = p^{-1}(S_{\Berko} ) \]
\end{lemme}
\begin{proof}

Let $m=(m_1,\ldots,m_n) \in M^n$, and 
let $x=|\cdot|_m :=p(m) \in \Bn_L$ the associated point in $\Bn_L$. 
Then there exists a unique non-Archimedean embedding 
$\H(x) \to L$ such $x_i=m_i$ where 
we consider the natural maps 
$\TnK \to \H(x) \to L$ and $x_i$ (resp. $m_i$) are the images of 
$X_i$. 
Hence, for 
$f\in \Ldan$, 
$|f(x)|=|f(m)|$, and 
$m$ satisfies $\varphi$ if and only if 
$x$ does.
\end{proof}

\begin{defi}
We will say that a map 
$\pi : L^n \to L^d$ is a coordinate projection if there exists a sequence 
$1\leq i_1 < \ldots < i_d \leq n$ such that $\pi$ is defined by 
\[ \pi(x_1,\ldots,x_n) = (x_{i_1}, \ldots, x_{i_d} ).\] 
\end{defi}

\begin{defi}
Let $X \subset \Kon$ and $Y\subset \Kom$ be subanalytic sets. 
We say that a map $f : X \to Y$ is a \textbf{subanalytic map}\index{subanalytic!map} if 
its graph $\Graph(f):= \{(x,y)\in X\times Y \st y=f(x)\}$ is a subanalytic set 
of $\Kov{m+n}$. 
\end{defi}
We want to list some elementary remarks and compatibility results about all these 
definitions. 

\begin{enumerate}
\item We fix $X \subset \Kon$ and $Y\subset \Kom$ some subanalytic sets
and  $f : X \to Y$ a subanalytic map. 
Let $K \to L$ be some non-Archimedean algebraically closed extension. 
Then we can naturally define a subanalytic map 
$f_L : X(L) \to Y(L)$ such that the following diagram commutes. 
\[ \xymatrix{
X \ar[r]^f \ar[d] & Y \ar[d] \\
X(L) \ar[r]_{f_L} & Y(L)
}\]
Indeed, $\Graph(f)$, the graph of $f$, is a subanalytic set of $\Kov{m+n}$ which satisfies the conditions 
\begin{flalign*}
&\Graph(f) \subset X \times Y  &\\
&\forall x\in X, \ \exists ! y\in Y \st  (x,y)\in \Graph(f) &
\end{flalign*} 
These two conditions are  equivalent to say that $\Graph(f)$ is the graph of a function 
$X \to Y$. 
Thanks to the uniform quantifier elimination theorem, these two conditions are still satisfied when we pass to 
$L$. 
\begin{flalign*}
&\Graph(f)(L) \subset X(L) \times Y(L) & \\
&\forall x\in X(L), \exists ! y\in Y(L) \st (x,y)\in \Graph(f)(L)  &
\end{flalign*}
Hence $\Graph(f)(L)$ is a subanalytic set of 
$\Kov{m+n}$ which defines a subanalytic map  $f_L : X(L) \to Y(L)$. 
\par 
Using exactly the same argument, we can prove that 
$f$ is injective (resp. surjective, bijective) if and only if 
$f_L$ is injective (resp. surjective, bijective). Also, with the same ideas, 
$f$ is continuous if and only if $f_L$ is continuous.
\item Let now 
$S \subset \Kov{m+n}$ and $T\subset \Kom$ be some subanalytic sets and 
let us assume that if $(x_1,\ldots,x_{m+n}) \in S$, then $(x_1,\ldots,x_m)\in T$. 
We then introduce the map 
$\pi : S \to T$ which is the projection on the first $m$-coordinates.
It is then clear that we obtain an associated map 
\[ \xymatrix{
S_\Berko \ar[r]^{\pi_\Berko} & T_{\Berko} 
} \]
which is actually induced by the map 
$\B^{m+n}_K \to \B^m_K$ corresponding to the projection on the first m coordinates. 
We want to prove that if $\pi$ is surjective (resp. injective) then 
$\pi_\Berko$ is surjective (resp. injective).  
\par If $\pi$ is surjective then $\pi_\Berko$ is also surjective. 
To see it, take $y\in T_\Berko$. 
Let then $K \to L$ be some non-Archimedean and algebraically closed extension 
and $(y_1,\ldots,y_m) \in T(L)$ such that 
$y = p(y_1,\ldots,y_m)$ where 
$p : L^n \to \B^m_K$ is the map introduced in \ref{remext}. 
We then obtain the commutative diagram 
\[ \xymatrix{
S(L)  \ar[r] \ar[d]_{\pi_L} & S_\Berko \ar[d]^{\pi_\Berko} \\
T(L) \ar[r] & T_\Berko 
} \]
Since $\pi$ is surjective, $\pi_L$ is also surjective, hence we can find 
$(x_1,\ldots,x_{m+n})$ some antecedent of $y$ by $\pi_L$. Then if 
$x$ is the image of $(x_1,\ldots,x_{m+n})$ in $S_\Berko$, by assumption 
$\pi_\Berko(x)=y$. 
\par Let us now assume that $\pi$ is injective.
Let then $y\in T_\Berko$ and $x,x' \in S_\Berko$ such that 
$\pi_\Berko(x)= \pi_\Berko(x') = y$. 
We obtain the diagram of non-Archimedean fields:
\[ \xymatrix{
 & \H(y) \ar[ld] \ar[rd] & \\
 \H(x) & &\H(x')
 }\]
 We can then find an algebraically closed non-Archimedean field $L$ which completes this diagram 
 \[ \xymatrix{
 & \H(y) \ar[ld] \ar[rd] & \\
 \H(x)\ar[rd] & &\H(x') \ar[ld] \\
 & L &
 }\]
Let us then write $(y_1,\ldots,y_m) \in \H(y)^m$ the evaluations of the $X_1,\ldots,X_m$ in $\H(y)$, and $(x_1,\ldots,x_{m+n})\in \H(x)$ the evaluations of $X_1,\ldots,X_{m+n}$ in 
 $\H(x)$ and similarly $(x'_1,\ldots,x'_{m+n})\in \H(x')$ the evaluations of $X_1,\ldots,X_{m+n}$ in 
 $\H(x')$. 
 By definition of $\pi$, if we work in $L$, we have 
 $y_i=x_i=x'_i$ for $i=1\ldots m$. 
 Hence $(x_1, \ldots,x_{m+n})$ and $(x'_1,\ldots,x'_{m+n})\in S(L)$ and by assumption 
 $p_L(x_1, \ldots,x_{m+n}) = p_L(x'_1, \ldots,x'_{m+n})=(y_1,\ldots,y_m)$. 
 Since $p_L$ is injective, $(x_1, \ldots,x_{m+n})=(x'_1, \ldots,x'_{m+n})$, hence 
 since $x\in S_\Berko$ corresponds to the image of 
 $(x_1,\ldots,x_{m+n})$ by the map $(L^\circ)^{m+n}$ and similarly for $x'$, it follows that $x=x'$.
\item According to the previous point, if we are in the situation 
(1) with a subanalytic map $f : X \to Y$, then since the projection
$\Graph(f) \to X $ is a bijection, according to (2) we deduce that the induced map 
 $\Graph(f)_\Berko \to X_\Berko$ is bijective. Hence this allows to define a natural map 
 $f_\Berko : S_\Berko \to T_\Berko$. 
 According to (2), $f_\Berko$ is injective (resp. surjective, bijective) if and only if 
$f$ is.
\end{enumerate}

\subsection{Proving general properties with Berkovich points}
\label{sec:gen_prop}
If $K \to L$ is a non-Archimedean extension and 
$y=(y_1,\ldots,y_n)\in L^n$, we will denote by 
$K(y) \subset L$ the subfield of L $K(y_1,\ldots,y_n)$ which is generated by $y_1,\ldots,y_n$. 
By definition, it then satisfies 
$K \subset K(y) \subset L$.
\begin{defi}
\label{defidimcorps}
If $K \to L$ is an extension of non-Archimedean field, the dimension of $L$ over $K$ is 
\[  d(L/K) = \text{tr deg} (\widetilde{L} / \widetilde{K} ) + 
\dim_\Q \Q \otimes_\Z  (|L^*|/ |K^*|) .\]
We will of often use that if $K \to L \to M$ are extension of non-Archimedean fields 
\begin{equation}
\label{transdeg}
d(M/K) = d(M/L)+d(L/K).
\end{equation}
We will also use that 
\begin{equation}
d(\widehat{L^\alg}/K)=d(\widehat{L}/K) = d(L/K).
\end{equation}
\end{defi}

There is a good dimension theory for $k$-analytic spaces (see \cite[p~34]{Berko90} and 
\cite{Duc07}). 
This dimension has a simple interpretation in terms of the 
definition \ref{defidimcorps}. Indeed if $X$ is 
a $K$-analytic space \cite[2.5.2]{Berko93}:
\index{dimension!of a $k$-analytic space}
\begin{equation}
\label{eq:dimBerko}
 \dim (X) = \sup_{x\in X} d( \mathcal{H}(x) /K).
 \end{equation}
We will show below that a similar result also holds 
for subanalytic sets (\ref{theodimsub} (5)). 

\begin{rem}
We want to give examples which explain why some attempts of replacing $d( \mathcal{H}(x) /K)$ by other numbers do not work. \par 
Let us try to replace $d( \mathcal{H}(x) /K)$ by $\td(\H(x) /K)$. If $\eta \in \B_K$ is the Gauss point, then 
$\td( \H(\eta) / K) = +\infty$. This can be seen because 
$K \subset K\langle T \rangle \subset \H(\eta)$ and $K\langle T \rangle$ 
contains infinitely many algebraically independent elements over K. 
For instance in mixed characteristic, the functions 
$T \mapsto exp(\lambda \cdot T)$ for some $| \lambda| \leq p^{-\frac{1}{p-1}}$ which are $\Q$-linearly independent 
gives such a family. \par 
When $X \subset \Bn_K$ is an affinoid space, one might try to consider the sup over non-Archimedean extensions $K \to L$ 
\[ \sup_{(x_i) \in X(L)} \td(K(x)/K). \]
Contrary to the previous attempt, this would give a finite number, but this would not be independent of the 
presentation of $X$.  
For instance, let $f\in K \langle T \rangle$ be a series such that 
$T$ and $f(T)$ are algebraically independent over $K$ and $\|f\| \leq 1$. 
And let $X \subset \B^2_K$ be the set defined by the equation $y=f(x)$. 
Then $\B_K \simeq X$ via the map $x \mapsto (x,f(x))$. 
If $\eta \in \B_K$ is the Gauss point (any point of type 2,3 or 4 would work), with the 
identification 
$K \subset K\langle T \rangle \subset \H(\eta)$, the 
point $(T,f(T)) \in X( \H(\eta) )$ and $\td ( (K(T,f(T) ) ) =2$ which is not the value 
we expect.
In other words, let $X \subset \Bn_K$ be some presentation of some affinoid space.
If $x\in X$ and if $\underline{x}:=(x_i) \in \H(x)^n$ is the coordinate of $x$ induced by the presentation of $X$, then 
$\td(K(\underline{x}) /K) $ is not independent of the presentation of $X$.
\end{rem}

\begin{lemme}
\label{lemmeimagelesspoint}
Let $X \subset \Kon$ and $Y\subset \Kom$be subanalytic sets.
Let $f:X \to Y$ be a subanalytic map.
Let $x \in X_{\Berko}$ and $y := f_\Berko (x) \in Y_\Berko$. 
Then 
$ d(\H(y)/K) \leq d(\H(x)/K)$. 
\end{lemme}
Of course, if $f$ is a morphism of $k$-affinoid spaces, then this is true because 
there is a map of non-Archimedean fields $\H(y) \to \H(x)$. 
But in our context where $f$ is just a subanalytic map, 
there is no given map $\H(y) \to \H(x)$.
\begin{proof}
Let us consider $(x_1,\ldots,x_n) \in X(\widehat{\H(x)^\alg})$ where for each $i$, 
$x_i\in \H(x)$ is the evaluation of 
the coordinates $X_i$ in $\H(x)$. 
Since we have a commutative diagram
\[ 
\xymatrix{
X(\widehat{\H(x)^\alg}) \ar[r]^f \ar[d]_p & Y(\widehat{\H(x)^\alg}) \ar[d]^p \\
X_\Berko \ar[r]_{f_\Berko} & Y_\Berko }
\]
$f(x_1,\ldots,x_n) \in ((\widehat{\H(x)}^\alg)^\circ)^m$. 
Since $y = p(f(x_1,\ldots,x_n))$, 
it follows that there is a non-Archimedean extension 
$\H(y) \to \widehat{\H(x)^\alg}$ and the result follows because 
\[d(\H(y)/K) \leq d(\widehat{\H(x)^\alg} /K) = d(\H(x)/K).\]
\end{proof} 

\begin{defi}
If $c\in K^n$, and $(r_1,\ldots,r_n)\in \R_+^n$, we define 
$\eta_{c,(r_1,\ldots,r_n)} \in \mathbb{A}^{n,an}_K$ by the formula 
\[
\begin{array}{cccc}
\eta_{c,(r_1,\ldots,r_n)} : & K[X] & \to & \R_+ \\
 & \displaystyle  f= \sum_{\nu=(\nu_1 \ldots \nu_n) \in \N^n} a_\nu \prod_{i=1}^n(X_i-c_i)^{\nu_i} & \mapsto & 
 |f(\eta_{c,(r_1,\ldots,r_n)})| = \displaystyle \max_{\nu \in \N^n} |a_\nu|r^{\nu}
\end{array}
 \]
where $r^\nu = r_1^{\nu_1} \cdots r_n^{\nu_n}$. \par 
Also if $f= \sum_{\nu \in \N^n} f_{\nu} X^\nu \in K[X_1,\ldots,X_n]$, we set 
\[\|f\| = \max_{\nu \in \N^n} |f_\nu|.\]
Then $\| f\| = \eta_{(0, (1,\ldots,1))}$.
\end{defi}
\begin{lemme}
\label{lemmegauss}
Let $K \hookrightarrow L$ be a non-Archimedean extension.
Let $x_1,\ldots,x_s \in L^*$ such that the real numbers
$r_i := |x_i|$, are $\Q$-linearly independent in $\Q \otimes_\Z (\R^*_+/|K^*|)$.
Let $(y_1,\ldots,y_t) \in L^\circ$ such that the $|y_i|=1$ and 
$\widetilde{y_i}$ are  algebraically 
independent over $\widetilde{K}$. 
Then 
the the image of 
$(x_1,\ldots, x_s,y_1,\ldots, y_t) \in L^{s+t}$ in $\mathbb{A}^{s+t,an}_K$ through 
$p:(L^\circ) \to \B^n_K$ is 
$\eta_{0,(r_1,\ldots,r_s,1,\ldots,1)}$.
\end{lemme}
\begin{proof}
We set $K[X,Y]:=K[X_1,\ldots,X_s,Y_1,\ldots,Y_t]$. \par 
Step 1. We claim that if 
\[ \ g= \sum_{\mu \in \N^t} g_{\mu} Y^\mu \in K[Y],\]
then
\[ 
|g(y)| = \|g\|,\]
where  $\| g\| = \max_{\mu\in \N^t} |g_\mu |$.  
To prove this we can assume that $\|g\|=1$, i.e. that $g\in K^\circ[Y] \setminus K^{\circ \circ}[Y]$. 
We can then consider $\tilde{g} = \sum_\mu \tilde{g_\mu}Y^\mu \in \widetilde{K}[Y]$ which is non zero. 
Now $g(y) \in L^\circ$ and in $\tilde{L}$, 
\[\widetilde{g(y)} =  \sum_{\mu \in \N^t} \widetilde{g_\mu} \widetilde{y^\mu}.\]
Since by assumption the $\tilde{y_i}$'s are algebraically independent over $\widetilde{K}$,  the right 
hand side must be non zero, so $\widetilde{g(y)} \neq 0$. So $|g(y)| =1 = \|g\|$. \par 
Step 2. We now consider some element 
\[f = \sum_{\nu \in \N^s} X^{\nu}f_{\nu}(Y) \in K[X,Y], \] 
where each $f_{\nu} \in K[Y]$.
According to step 1, $|x^{\nu} f_{\nu}(y)| = r^{\nu}\|f_{\nu} \|$, then for 
$\nu \neq \mu$, with 
$f_\nu \neq 0$, $|x^{\nu} f_{\nu}(y)| \neq |x^{\mu} f_{\mu}(y)|$. 
It then follows from the triangle ultrametric inequality that 
\[|f(x,y)| = \max_{\nu \in \N^s} r^{\nu} \|f_{\nu} \|\] which is precisely 
$|f(\eta_{0,(r_1,\ldots,r_s,1,\ldots,1)} )|$.
\end{proof}

\begin{prop}
\label{proplocconst}
Let $K \hookrightarrow L$ be a non-Archimedean extension, and 
$y=(y_1,\ldots,y_n) \in L^n$. 
Let $p : L^n \to (\An_K)^\an$ as in remark \ref{remext} and 
$x:=p(y)$. Then if 
$d(\H(x) /K) = n$, there exists a neighborhood $V$ of $y$ in 
$L^n$ such that $p(V)= \{x\}$.
\end{prop}
Such points $x$ of $(\An_K)^\an$ are sometimes called \textbf{Abhyankar points}\index{Abhyankar points}. 
We refer to \cite[4.1]{PoiAng} for another appearance of these points.
\begin{proof}
First note that 
$d(K(y)/K) = d(\H(x)/K)$ because $\mathcal{H}(y) \simeq \widehat{K(y)}$. 
Let $s := \text{tr.deg.}(\widetilde{K(y)} /\widetilde{K})$ 
and $t:= \dim_\Q ( \Q \otimes_\Z |K(y)^*|/|K^*| )$. \par 
Then $n=s+t$. 
Hence for each $i=1\ldots n$ one can introduce some fractions
$F_i = \frac{P_i}{Q_i} \in K(X)$ such that for all $i=1,\ldots,n$, 
\begin{align*} 
 Q_i(y) \neq 0. \\
 \left\{ \left| \frac{P_i(y)}{Q_i(y)} \right| \right\}_{i=1 \ldots t} 
 \text{ are } \Q\text{-linearly independent in} \ \Q \otimes_\Z \big( |K(y)^*|/|K^*| \big).  \\
 \text{For} \ i=t+1 \ldots n,  \ \left| \frac{P_i(y)}{Q_i(y)} \right| =1, \ \text{and} \\ 
 \left\{ \widetilde{ \left( \frac{P_i(y)}{Q_i(y)} \right) } \right\}_{i=t+1 \ldots n} \text{are algebraically 
independent over} \ \widetilde{K}.
\end{align*}
We set  $r_i:=   \left| \frac{P_i(y)}{Q_i(y)} \right|$. \par 
Let now $\mathcal{U}$ be the affine subset of $\An_K$ defined by 
$\mathcal{U} = \{z\in \An_K \st Q_i(z)\neq 0 , \ i=1\ldots n \}$.
Then 
\[F=(F_1,\ldots,F_n) : \mathcal{U} \to \An_K \] is a regular map and we obtain the following commutative diagram: 
\[
\xymatrix{
\mathcal{U}^{an} \ar[r]^{F^{an} } & \mathbb{A}^{n,an}_K \\
\mathcal{U}(L) \ar[u]^p \ar[r] & L^n  \ar[u]_p
}
\]
Then by assumption $y\in \U(L)$, and according to lemma \ref{lemmegauss}, 
\[F^{an}(p(y)) = F^{an}(x) = \eta_{0,(r_1,\ldots,r_s,1,\ldots,1)} .\]
It is a standard fact that
$d(\mathcal{H}(\eta_{0,(r_1,\ldots,r_s,1,\ldots,1)})/K)=n$. 
Hence since $\dim(\U^\an)=n$ (we mean here the dimension as a $k$-analytic space), for any 
point $u$ in the fiber 
$ (F^\an)^{-1}(\eta_{0,(r_1,\ldots,r_s,1,\ldots,1)})$, 
\[d(\mathcal{H}(u) / \mathcal{H} (\eta_{0,(r_1,\ldots,r_s,1,\ldots,1)}) )=0. \]
According to the formula \eqref{eq:dimBerko}, 
we deduce that  $ (F^\an)^{-1}(\eta_{0,(r_1,\ldots,r_s,1,\ldots,1)})$ is 
a $0$-dimensional $\H(\eta_{0,(r_1,\ldots,r_s,1,\ldots,1)})$~-analytic space. 
Hence as a topological space,  
$ (F^\an)^{-1}(\eta_{0,(r_1,\ldots,r_s,1,\ldots,1)})$ is discrete 
(thanks to the definition of the dimension of a $k$-analytic space \cite[p.34]{Berko90}).
In particular, $x$ is an isolated point in its fiber $ (F^\an)^{-1}(\eta_{0,(r_1,\ldots,r_s,1,\ldots,1)})$.\par 
By a simple continuity argument, there exists a neighborhood  $V_1$ of $y$ in 
$\U(L)$ such that  all points $v\in V_1$ satisfy exactly the same conditions than $y$ listed above. By this we mean in particular that 
\begin{align*}
\text{For} \ i=1 \ldots t, \  
\left| \frac{P_i(v)}{Q_i(v)} \right|  = \left| \frac{P_i(y)}{Q_i(y)} \right|. \\
\text{For} \ i=t+1\ldots n, \ 
\widetilde{ \frac{P_i(v)}{Q_i(v)} } =\widetilde{ \frac{P_i(y)}{Q_i(y)} }.
\end{align*} 
Hence, the same argument as for 
$y$ can be applied to $v$. Namely, for all $v\in V_1$,  
$F^{an}(p(v)) = \eta_{0,(r_1,\ldots,r_t,1,\ldots,1)}$. 
But since $p(y)=x$, $p$ is continuous, and since $x$ is isolated in its fibre $(F^{an})^{-1}( \eta_{0,(r_1,\ldots,r_t,1,\ldots,1)})$, 
there exists a neighborhood $V \subset V_1$ of $y$  such that 
$p(V)=\{x\}$.
\end{proof}

\begin{theo}
\label{theodimsub}
Let $S$ be a nonempty subanalytic set of 
$\Kon$. The following numbers are equal: 
\begin{enumerate}
\item 
The greatest $d$ for which there exists a coordinate projection  
$\pi : \K^n \to \K^d$ such that 
$\pi (S)$ has nonempty interior.
\item
The greatest $d$ for which there exists a coordinate projection 
$\pi : \K^n \to \K^d$ such that 
$\pi (S)$ is somewhere dense (that is to say its adherence has nonempty interior).
\item 
\[d = \max_{x\in S_{\Berko}} \text{tr deg}(\widetilde{\mathcal{H}(x)} /\widetilde{K}) .\]
\item The greatest $d$ such that there is a non-Archimedean extension 
$K\to L$ and a point $x\in S(L^{alg})$ such that 
tr deg$(\widetilde{K(x)}/\widetilde{K} ) = d$.
\item
\[d = \max_{x\in S_{\Berko}} d(\mathcal{H}(x) /K) .\]
\end{enumerate}
We define the dimension of $S$ to be the number $d$ which satisfies these 
equivalent properties. 
If $S$ is empty we set $\dim(S) = - \infty$. 
\end{theo}

\begin{proof} Here $(i) \Rightarrow (j)$ will mean that the integer $d_i$ defined in $(i)$ is smaller than $d_j$, the integer defined in $(j)$. \par
$(1) \Rightarrow (2)$ is clear.\par
$(2) \Rightarrow (4)$.
We can introduce
$r := |\lambda| \in |K^*|$ and 
$c\in \pi(S)$ such that 
$B(c,r)$, the closed ball of center $c$ and radius $r$ in $\Kod$, is included in 
$\overline{ \pi(S)}$. 
Let $\K \to L$ be a complete non-Archimedean extension with $L$ algebraically closed, and such that 
there exists $(y_1,\ldots,y_d) \in L^d$ such that 
$|y_i|=1$ and the
$\widetilde{y_i}$ are algebraically independent over $\widetilde{\Ka}$. 
Now, according to the uniform quantifier elimination theorem, 
we still have\footnote{If $T$ is a subanalytic set, $\overline{T}$ is also subanalytic 
because the closure can be defined by a first order formula.} 
$B(c,r)_L \subset  \overline{ \pi(S(L))}$. 
In particular, $c+ \lambda y \in \overline{ \pi(S(L))}$.
But if we replace the $y_i$'s by very close elements, we will still have that 
$|y_i|=1$ and the
$\widetilde{y_i}$ are algebraically independent over $\widetilde{\Ka}$. 
Hence we can in fact assume that $c+ \lambda y \in  \pi(S(L))$. 
Hence by definition, we can complete the $y_i$'s, with some 
$y_{d+1},\ldots,y_{n}$, such that 
$(y_1,\ldots,y_n)\in S(L)$. 
Now, 
$\td(\widetilde{(K(c+\lambda y)} / \widetilde{\Ka})=d$. 
Hence $\td(\widetilde{K(c + \lambda y)}/K)=d$ because in the following diagram 
\[
\xymatrix{
\Ka(c + \lambda y) & \Ka \ar[l] \\
K(c + \lambda y) \ar[u] & K \ar[l] \ar[u]
}
\]
the horizontal inclusions have a residual transcendental degree $0$. 
\par 
$(3) \Leftrightarrow (4)$ follows from the definition of $S_{\Berko}$.\par
$(3) \Rightarrow (5)$ is clear.\par
$(5) \Rightarrow (1)$
Let $x\in S_{\Berko}$ such that 
$d(\H(x)/K)=d$. 
Then there exists $K \to L$ a complete non-Archimedean algebraically closed extension, and 
$y\in L^n$ such that 
$p(y) = x$ where $p :L^n \to \Bn_K$ is as in remark \ref{remext}. For instance take 
$L = \widehat{\H(x)^{alg}}$. \par
Then $d(K(y)/K) = d(\H(x)/K)=d$,
hence there exists 
$1 \leq i_1 < \ldots < i_d\leq n$ such that 
$d(K(y_1, \ldots, y_n) /K)=d( K(y_{i_1},\ldots,y_{i_d})/K)=d$.
Let then 
$\pi : L^n \to L^d$ be the coordinate projection along the coordinates 
$i_k, k=1 \ldots d$ and 
let $z:=(y_{i_1},\ldots, y_{i_d} )=\pi(y) \in \pi(S)(L)$.
Then d$(K(z)/K)=d$. Then according to proposition 
\ref{proplocconst}, there exists a neighborhood 
$V$ of $z$ in $L^d$ such that 
$p(V)=\{p(z) \}$ where 
$p : L^d \to \mathbb{B}^d_K$. 
Since $z\in \pi(S)_{\Berko}$, according to lemma \ref{lemmetaut}, 
$V \subset \pi(S)(L)$. 
Hence $\pi(S)(L)$ contains $V$, hence has nonempty interior. Since having 
nonempty interior is definable with a first order formula, 
according to the 
uniform quantifier elimination theorem, 
$\pi(S)$ has also non empty interior.
\end{proof}

\begin{lemme}
The dimension of subanalytic sets is invariant by scalar extension.
By this we mean the following. 
Let $K \to L$ be an algebraically closed non-Archimedean extension, and let $S$ be a subanalytic set 
of $\Kon$, and $S(L)$ the associated subanalytic set in $(L^\circ)^n$. Then 
$\dim(S) = \dim(S(L))$.
\end{lemme}
\begin{proof}
This follows from the uniform quantifier elimination theorem and the fact 
that condition $(1)$ in the above theorem can be 
expressed by a first order formula.
\end{proof}

\begin{prop}
\label{propsubdim0}
Let $S \subset \Kon$ be a subanalytic set.
\begin{enumerate}
\item 
$ \dim(S)=0 \Leftrightarrow S$ is nonempty and finite.
\item 
$\dim(S) =n \Leftrightarrow S$ has non empty interior $\Leftrightarrow S$ is somewhere dense.
\item 
$ \dim(S) < n \Leftrightarrow S$ has empty interior $\Leftrightarrow S$ is nowhere dense.
\end{enumerate}
\end{prop}
\begin{proof}
\begin{enumerate}
\item 
If $S$ is finite, its dimension is clearly $0$ according to the characterization (1) of 
theorem \ref{theodimsub}. \par
Conversely, if $\dim(S)=0$, then for $i=1\ldots n$, let $\pi_i$ be
the coordinate projection along the $i$-th coordinate 
$\pi_i  : \Kon \to \Ko$. Then $\pi_i(S)$ has empty interior.
But $\pi_i(S)$ is a subanalytic set of $\Ko$ and they are either finite or 
have nonempty interior (according to their description 
 as Swiss cheeses). 
Hence each $\pi_i(S)$ must be finite, so $S$ itself is finite.
\item and (3) are consequences of the characterizations $(1)$ of the 
theorem \ref{theodimsub} and $(2)$ for $d=n$. 
\end{enumerate}
\end{proof}

\begin{rem}
 \label{rem:conv_dim}
 If $S\subset \Kon$ is subanalytic, and $d= \dim(S)$, then for each integer $k$ such that 
 $0 \leq k \leq d$, there exists $x\in S_\Berko$ such that 
 $d ( \H(x) / K)=k$. 
\end{rem}

\begin{prop}
\label{propdimunionsub}
If $S_i, \ i=1\ldots N$ are subanalytic sets of $\Kon$, 
\[\dim(\bigcup_{i=1}^N S_i) = \max_{i=1 \ldots N} \dim(S_i).\]
\end{prop}
\begin{proof} This follows from the characterisation 
$(5)$ in theorem \ref{theodimsub} in term of 
points in Berkovich spaces.
\end{proof}
More generally, let $\{S_i\}_{i\in I}$ be an arbitrary family of subanalytic sets of 
$\Kon$ and let us assume that 
$\bigcup_{i\in I} (S_i)_\Berko$ is a subanalytic set of $\B^n_K$ then 
$\dim(\bigcup_{i\in I} (S_i)_\Berko) = \max_i \dim((S_i)_\Berko)$. 
This property is false if we remove the $_\Berko$, as can be seen if we take all 
the singletons $\{x\}$ of $\Ko$ for instance. 
\begin{prop}
\label{defS(i)sub}
Let $S$ be a subanalytic subset of $\Kov{n+m}$. 
For $x\in \Kon$, let 
\[S_x:= \{y\in \Kom \st (x,y)\in S\} 
= \pi^{-1}(x)\]
 where 
$\pi : \Kov{n+m} \to \Kon$ is the projection on the first $n$ coordinates. 
For each integer $i$, let 

\[S(i) :=  \{ x\in \Kon \st \dim(S_x)=i \} \]
\[S^{(i)} := \{ (x,y)\in S \ \text{with} \ x\in \Kon, \ y\in \Kom \st \dim(S_x)=i \} = \pi^{-1} (S(i)) .\]
Then $S(i)$ and $S^{(i)}$ are subanalytic sets. Moreover, this is compatible with 
scalar extension, that is to say, 
if $K\subset L$ is an extension of complete fields
$S(i)(L^{alg}) = S(L^{alg})(i)$ and 
$S^{(i)} (L^{alg}) = S(L^{alg})^{(i)}$. 
\end{prop}

\begin{proof}
If $T$ is a subanalytic set of $\Kom$, according to the characterisation $(1)$ for the dimension 
in theorem \ref{theodimsub}, the property 
$\dim(T) \geq d$ is expressible at the first order. 
Namely, the property 
"$U$ has non empty interior" can be formulated by the formula :\\
"there exists a point a point $c\in U$, a radius $\gamma \in \Gamma$ such that 
$U$ contains the ball of center $c$ and radius $\gamma$."\\
Now $\dim(T) \geq d$ if and only if there exists a coordinate projection 
$\pi : \Kom \to \Kov{d}$ such that 
$\pi(T)$ has non empty interior, and since there are only finitely many coordinate projections, the properties 
$\dim(T)  \geq d$ and $\dim(T) =d$ are well definable with a first order formula. 
This implies that $S(i)$ and hence  $S^{(i)}$ are subanalytic. 
The fact that it behaves well if we increase the field 
$K$ is then a consequence of the uniform quantifier elimination theorem.
\end{proof}

\begin{prop}
\label{propdimfamily}
With the above notations, 
\[\dim\big(S^{(i)} \big) = \dim \big(S(i) \big)+i .\] 
As a corollary, 
\[ \dim(S) = \max_{i\geq 0} \dim(S(i)) +i .\]
\end{prop}
\begin{proof}
Let us prove first that 
$\dim(S^{(i)}) \leq \dim(S(i))+i$. To do that we use the characterization 
(5) of theorem \ref{theodimsub}.
Let $K \subset L$ be an algebraically closed non-Archimedean extension,  
$(x,y) \in S^{(i)} (L)$. 
Then $x\in S(i)(L)$ hence 
\begin{equation}
\label{ineg1sub}
d(K(x)/K) \leq \dim (S(i)) 
\end{equation}
Let $M:=\widehat{ K(x) } \subset \widehat{L}$ seen as sub-valued field of $L$. 
Then $y\in S_x(L^{alg})$ hence by assumption 
\[d(M(y)/M) \leq i = \dim(S_x). \]
Now in the following diagram
\[
\xymatrix{
M(x,y) = M(y) & K(x,y) \ar[l] \\
M \ar[u] & K(x) \ar[l] \ar[u]
}
\]
the two horizontal field inclusions satisfy $d(\cdot / \cdot)=0$, hence 
\begin{equation}
\label{ineg2sub}
 d(K(x,y)/K(x) )= d (M(y)/M)  \leq i .
\end{equation}
Hence the inequalities \eqref{ineg1sub} and \eqref{ineg2sub} imply that 
\[ d (K(x,y)/K) = d(K(x,y)/K(x) ) + d(K(x)/K) \leq   \dim (S(i)) +i \] 
hence 
$\dim (S^{(i)}) \leq \dim (S(i))+i$. \par
Let us prove conversely that 
$\dim (S(i)) +i \leq \dim (S ^{(i)})$.
Let 
$K \to L$ be an algebraically closed  non-Archimedean extension and 
$x\in S(i)(L)$ such that
\begin{equation}
\label{in1}
 d(K(x)/K) = \dim (S(i)). 
\end{equation} 
We can assume that 
$L = \widehat{K(x)^{alg}}$.\\
Since  
$\dim (S_x)=i$, there exists 
$L \to M$ an algebraically closed non-Archimedean extension, and
$y \in S_x(M)$ such that 
$d(L(y)/L) =i$. 
Then $(x,y) \in S^{(i)}(M)$. 
For the same reason as above, 
\begin{equation}
\label{in2}
 d(K(x,y)/K(x)) = d(L(y)/L)=i
 \end{equation} 
 because $L=\widehat{K(x)^{alg} }$. 
Hence with \eqref{in1} and \eqref{in2}:
\[\dim (S^{(i)}) \geq d(K(x,y)/K) = d(K(x,y)/K(y) + d(K(x)/K) = \dim (S(i)) +i. \]
The corollary now follows from the fact that 
$S = \cup_i S^{(i)}$ and proposition \ref{propdimunionsub}.
\end{proof}

\begin{prop}
\label{propproduit}
If $S$ (resp. $T$) is a subanalytic set of $\Kon$ (resp. $\Kom$), then 
$S\times T$ is a subanalytic set of $\Kov{n+m}$ and 
\[\dim(S\times T) = \dim(S) +\dim(T).\]
\end{prop}
\begin{proof} 
This follows directly from the characterization (1) of the dimension in theorem \ref{theodimsub}.
\end{proof}

\begin{prop}
\label{propinvar}
Let $f:X \to Y$ be a subanalytic map of subanalytic sets of $\Kon$ and $\Kom$. 
\begin{enumerate}
\item 
$\dim (f(X)) \leq \dim(X)$.
\item If $f$ is injective, $\dim (f(X)) = \dim (X)$.
\item If $f$ is bijective $\dim(X)=\dim(Y)$.
\end{enumerate}
\end{prop}
\begin{proof}Since 2. and 3. are consequences of 1. so we only have  to prove 1. 
But 1. follows directly from lemma \ref{lemmeimagelesspoint} and the characterisation 5. of the dimension in theorem \ref{theodimsub}. 
\end{proof}  

\begin{rem}
 More generally, if $f: X \to Y$ is a subanalytic map between subanalytic sets, then 
 \[ \dim(S) = \sup_{y\in Y} ( d(\H(y) /K) + \dim_{\H(y)} ( f_{\Berko}^{-1}(y) )).\]
\end{rem}
\begin{rem}
\label{remdimsubqa}
More generally, we can define the dimension of subanalytic sets of a good strictly $k$-affinoid space. 
The invariance of dimension by subanalytic bijection (proposition \ref{propinvar} (3))  
is necessary to prove that this definition works well. This theory of dimension then satisfies all the properties mentioned above.
\end{rem}
 
\subsection{Closure and boundary}
\label{sec:closure}

\begin{lemme}
\label{lemme:adhA}
Let $S \subset (K^\circ)^n$ be a subanalytic set of dimension $d$, and let 
$x\in S_\Berko$ such that $d(\H(x)/K)=d$.
Let $K \to L$ be an algebraically closed non-Archimedean extension and 
$p:L^n \to \Anan$ the canonical map. 
Then $p^{-1}(x)$ is an open subset of $S(L)$.
\end{lemme}
Remark that this set might be empty.
\begin{proof}
Let $u\in L^n$ be a preimage of $x$ by $p$.
We can then consider a coordinate projection 
$\pi : \Anan \to (\mathbb{A}_k^d)^\an$ such that 
$y:= \pi(x)$ is an Abhyankar point.
Then $\pi^{-1}(y) \cap S_\Berko$ can be identified with a subanalytic set of 
$\B^n_{\H(x)}$ of dimension $0$, hence is finite.
Let us say that $\pi^{-1}(y)\cap S_\Berko = \{x,x_1,\ldots,x_m\}$.\par 
In addition, according to proposition \ref{proplocconst}, $p^{-1}(y)$ is open, and so 
$\pi_L^{-1} (p_i^{-1}(y))$ is open where 
\[ 
\xymatrix{
L^n \ar[r]^p \ar[d]^{\pi_L} & \B^n \ar[d]^{\pi} \\
L^i \ar[r]^{p_i} & \B^i
}\]
Then 
\[ S(L) \cap ( \pi_L^{-1}( p_i^{-1}(y)) = 
S(L) \cap (p^{-1}(x) \cup p^{-1}(x_1) \ldots \cup p^{-1}(x_n) ) \]
is  open in $X(L)$.
But since $p^{-1}(x_i)$ is closed for all $i$ (just because affinoid spaces are Hausdorff 
for the Berkovich topology), it follows that 
$X(L) \cap p^{-1}(x)$ is open in $X(L)$.
\end{proof}

\begin{lemme}
\label{lemme:semi-cont}
Let $K \to L$ be an extension. The map 
\[\begin{array}{ccc}
L^n & \to & \N \\
x & \mapsto & d(K(x) / K) 
\end{array}
\]
is lower semi-continuous. 
In other words for $d\in \N$, 
the sets $\{x\in L^n \st d(K(x) /K) \geq d\}$ are open. 
In other words again, for $d\in \N$, 
the sets $\{x\in L^n \st d(K(x) /K) \leq d\}$ are closed.  
\end{lemme}
\begin{proof}
First we remark that $d(K(x) / K) \geq d$ if and only if for some coordinate projection 
$p : L^n \to L^d$, 
$d(K(p(x))/K) =d$. \par 
From this it follows that it is enough to show that 
$\{x\in L^n \st d(K(x) /K) \geq n\}$ is open. 
But this has already been proved during the proof of prop \ref{proplocconst}.
\end{proof}

\begin{prop}
Let $S \subset (K^\circ)^n$ be a subanalytic set. 
Then $\overline{S}$ is also subanalytic and 
$\dim(\overline{S} \setminus S) < \dim(S)$.
\end{prop}

\begin{proof}
First, $\overline{S}$ is still subanalytic because 
the closure of a set in $\Kon$ can be expressed by a first order formula:
\[
\overline{S} = \{x \in \Kon \st \forall \varepsilon \in \Gamma \ \exists s \in S \
\text{such that} \ |x-s|\leq \varepsilon \}.\]
Step 1: First using lemma \ref{lemme:semi-cont}, we obtain that $\dim(\overline{S}) \leq \dim(S)$. \par 
Step 2: Let us assume that $\dim(\overline{S}\setminus S) = \dim(S) =d$. 
Let then $x\in \overline{S}_\Berko \setminus S_\Berko$ such that 
$d(\H(x) /K) =d$.
Let then $K \to L$ be some algebraically closed non-Archimedean extension such that 
$x\in p(L^n)$.
Then according to lemma \ref{lemme:adhA}, 
$p^{-1}(x)$ is an open subset of 
$\overline{S}(L)$ which is contained in 
$\overline{S}(L) \setminus S(L)$ which is absurd.
\end{proof}

\subsection{Bad functions in $C$-minimal structures}
\label{sec:Cminimal}
In \cite[sect.6, p.154]{HaMaCell} it is noted that when a complete first order theory satisfies the exchange property, 
then the rank is a \emph{good} notion of dimension for definable subsets.
Afterwards, the notion of a bad function is introduced. 
In our context (namely ACFV, or the analytic theory associated for a complete non-Archimedean field), 
this is a definable function 
\[f : K \to \{ \text{closed  balls  of strictly positive radius} \} \]
such that for some nonempty open ball $C \subset K$, $f_{|C}$ induces an isomorphism of 
quasi C-structures (see \cite[2.2]{HaMaCell}) with its image. \par 
This implies in particular that 
if $b\neq b'$ are elements of $C$, $f(b)$ and $f(b')$ are disjoint balls. 
Then if we set $X := \cup_{b\in C} f(b)$,  this is a definable set of $K$,  
and for each $x\in X$ there is a unique $b\in C$ such that 
$x\in f(b)$.
We then define $g: X \to C$ by $g(x) =b$.
This is a definable function, with infinite range and infinite fibers.
So according to proposition \ref{propdimfamily}, we should have $\dim(X)=2$, but this is impossible 
because $X \subset K$.
From this we conclude 
\begin{prop}
Let $T$ be the C-minimal structure associated with the analytic structure of a complete non-Archimedean field, or with ACVF. Then $T$ does not have bad functions.
\end{prop}
In \cite{HaMaCell} it is then proved that a $C$-minimal structure has the exchange property if and only if it has no bad function  \cite[prop. 6.1]{HaMaCell} and 
that  if it has no bad function, then the dimension as we have defined it in theorem  \ref{theodimsub} (1) coincides with the rank \cite[prop. 6.3]{HaMaCell}.

\section{Mixed dimension of definable subsets of $K^m \times \Gamma^n$}
\label{sec2}
In this section we will consider two settings: 
$K$ will be either 
\begin{itemize}
 \item an algebraically closed valued field 
\item or an 
algebraically closed non-Archimedean field.
\end{itemize}
Depending on these contexts, a definable set of $\KmGn$ will 
mean either
\begin{itemize}
 \item  a definable set of $\KmGn$ in the language  $\mathcal{L}_1$ considered in section \ref{subsubsec:ACVF} 
 \item or a subanalytic set of 
 $\Kom \times \Gamma^n$ defined in the language $\mathcal{L}_1^D$ considered in section \ref{subsubsec:an}.
\end{itemize}
We are going to define what relative cells $C \subset \KmGn$ are\footnote{Remind that we should rather consider $K^m \times (\Gamma_0)^n$ instead, 
but following the remark made in section \ref{subsec:warning}, we will ignore $0$.}. 
They will be the data 
of a definable subset $X \subset K^m$, $\icell$ a sequence of $0$'s and $1$'s, 
and $C$ will be 
a nice family of $\icell$'s parametrized by $X$. \par 
We will give a uniform treatment 
of the cell decomposition \ref{theomixedcelldec}, and of the mixed 
dimension \ref{theodimmixed} that we will build for definable 
subsets of $\KmGn$, so the proof for 
the ACVF and the analytic settings will be the same.
This is rendered 
possible by the analogous properties of the 
dimension of subsets of 
$K^n$ in the ACVF setting (see section \ref{subsection:dimACVF}) and 
of the dimension of subanalytic set of $\Kon$ in the subanalytic setting (see section \ref{sec1.2}). 
As for the definable subsets of $\Gamma^n$, 
they are the same in both theories (and have been described in section \ref{sec1.4}).

\subsection{Mixed cell decomposition theorem}
\label{sec2.1}

\begin{defi}
\label{deficontindefi}
Let $X \subset K^m \times \Gamma^n$ be definable set. We will set:
\[
 \mathcal{C}^\circ(X) := \{f : X \to \Gamma \st f \ \text{is continuous and definable} \} 
\cup \{\pm \infty \} .\] 
\end{defi}
We will need the following proposition.

\begin{defi}
\label{deficell}
Let $X \subset K^m$ be a nonempty definable subset. 
For $\icell \in \{0,1\}^n$ we define inductively  
\textbf{$X$-$\icell$-cells}, which are definable subsets of 
$X \times \Gamma^n$: 
\begin{enumerate}[i)]
\item if $n = 0$, then $X$is the only 
$X-\emptyset$-cell.
\item If 
$C \subset X \times (\Gamma)^n$ is an 
$X-\icell$-cell, and $f,g \in \mathcal{C}^\circ(X)$, such that 
for all $x \in C, \ f(x)<g(x)$, then 
\begin{align*}
\Graph(f) := \{ (x,\gamma)\in C \times \Gamma \st \gamma =f(x) \} \ \text{is an} \ 
X-(i_1,\ldots,i_n,0)\text{-cell}.   \\
]f,g[ := \{ (x,\gamma)\in C \times \Gamma \st f(x)<\gamma <g(x) \} \ \text{is an} \ 
X-(i_1,\ldots,i_n,1)\text{-cell}. 
\end{align*}

\end{enumerate}
\end{defi}

\begin{rem}
\label{remcell}
If $X \subset K^m$ is a definable set, 
$C$ is an 
$X-\icell$-cell, then for all 
$x\in X$, if we set 
\[C_x := \{ \gamma \in \Gamma^n \st (x,\gamma)\in C \} \]
then $C_x \subset \Gamma^n$ is an 
$\icell$-cell in the sense of the definition \ref{deficellGamma}, 
and one has to think of $C$ as a definable continuous family 
of $\icell$-cells parametrized by $X$.
Remark that if 
$(i_1,\ldots,i_n) \neq (j_1,\ldots,j_n)$, then a 
$X-(i_1,\ldots,i_n)$-cell and a 
$X-(j_1,\ldots,j_n)$-cell are necessarily different. This is true because anyway 
a $(i_1,\ldots,i_n)$-cell and a $(j_1,\ldots,j_n)$-cell are different. 
\end{rem}

\begin{defi}
\label{defdimcell}
Let $X\subset K^m$ be definable and $C\subset X\times \Gamma^n$ a 
$X-(i_1,\ldots,i_n)$-cell. Let 
$d_1 = \dim(X)$ and $d_2 = i_1+ \ldots+i_n$. 
\index{dimension!of an $X$-$(i_1,\ldots,i_n)$-cell}
We say that $C$ is a cell of dimension $(d_1,d_2) \in \N^2$.
\end{defi}

We make this first observation:
\begin{prop}
\label{propopencell}
Let $C \subset \KmGn$ be a cell. Then 
$\dim(C) = (m,n) \Leftrightarrow C$ has non-empty interior. 
In other words, $\dim(C) \neq (m,n) \Leftrightarrow C$ has empty interior.
\end{prop}
\begin{proof}
We first prove that 
$\Big( \dim(C) = (m,n) \Big) \Rightarrow C $ has non-empty interior. 
So let us consider $C$  an $X-(1,\ldots,1)$-cell with $\dim(X)=m$. 
Then according to proposition \ref{propsubdim0} (2), 
$X$ has non-empty interior, hence we can assume that $X$ is open. 
Then, using the definition of a cell, one proves by induction on $n$ that an 
$X-(1,\ldots,1)$-cell is open.\par 
Finally  $ \Big( C$ has non-empty interior $\Big) \ \Rightarrow \dim (C) = (m,n)$ is clear. 
\end{proof}

\begin{prop}
Let $S\subset K^m \times \Gamma^n \times \Gamma$ be a subanalytic set.
Assume that for all $z\in K^m \times \Gamma^n$ the fibre 
$S_z : =\{ ( \gamma \in \Gamma \st (z,\gamma) \in S \}$ is finite.
Then there exists an integer $N$ such that 
for all $z\in K^m \times \Gamma^n$, 
$\Card(S_z) \leq N$.
\end{prop}

\begin{proof}
Following the proof of 
proposition \ref{propfonccont}, we can assume that $S$ is defined by a 
formula $\varphi$ involving only the variables 
$(x,\alpha) \in K^m \times \Gamma^n$ and a conjunction of formulas 
\[ \{ (x,\alpha,\gamma) \in \Kom\times \Gamma^n \times \Gamma \st 
|f_i(x)| \alpha^{u_i} = \gamma^{a_i} \} \ i=1\ldots N \]
with $f_i : K^m \to K$ a definable function, 
$u_i \in \Z^n$, and $a_i \in \N^*$. \\
Hence, if $z=(x,\alpha)\in K^m \times \Gamma^n$,
\[ S_z \subset \Big\{ \sqrt[a_i]{ |f_i(x)| \alpha^{u_i} }  \Big\}_{i=1\ldots N} \]
hence $\Card(S_z)\leq N$.
\end{proof}

We now lists some properties which are the essential ideas of the proof of the 
cell decomposition theorem \ref{theomixedcelldec}. 
Note that this a complete adaptation of the proof 
\cite[th 4.2.11]{VDD}:

\begin{fact}[O-minimality in a nutshell]
\label{lemmebound}
Let $A \subset \Gamma$ be a definable subset.
\begin{enumerate}
\item 
There exists $M,N \in \N$, two sequences
\[a_1 <b_1 \leq a_2 <b_2 \leq a_3 < \ldots b_{N-1} \leq a_N <b_N \]
and $c_1<c_2 <\ldots <c_M$ of elements of $\Gamma$ 
(we allow $a_1= -\infty$ and $b_N=+\infty$), and 
some $\lceil_i ,\rfloor_i \in \big\{ ],[ \big\}$ for $i=1\ldots N$ such that 
\begin{align*}
 A = ( \bigcup_{i=1}^M \lceil_i a_i,b_i \rfloor_i ) \cup ( \bigcup_{j=1}^N \{c_j\} ) \\
 \text{for}  \ i=1\ldots N, \ j=1\ldots M, \ c_j \notin \{ a_i, b_i \} \\
 \text{if} a_i = b_{i+1}, \text{then}  \rfloor_i = [ \ \and \lceil_{i+1} = ].
 \end{align*}
This definition forces the sets $\lceil_i a_i,b_i \rfloor_i$ and $\{c_j\}$ to be the maximal intervals contained in $A$, hence are uniquely determined by A, in 
 particular the couple $(M,N)$ is determined by $A$, and we will say that $A$ is of type $(M,N)$.
 \item 
 Bd$(A) = \{a_i,b_i,c_j \ i=1\ldots N, \ j=1\ldots M\}$ where 
 Bd$(A) = \overline{A}\setminus A$. 
 \item 
 If $S \subset K^m \times \Gamma^n\times \Gamma$ is a definable set, 
 for $(M,N)\in \N^2$ let 
 \[S^{(M,N)} = \{ z\in K^m \times \Gamma^n \st S_z \ \text{is a definable set of} \ \Gamma \ \text{of type} \ (M,N) \}.\]
 Then 
 $S^{(M,N)}$ is definable.
 \item 
 Let $S \subset K^m \times \Gamma^n \times \Gamma$ be a definable set and $(M,N) \in \N^2$.
 If $z\in \SMN$, $S_z \subset \Gamma$ is by definition a definable set of $\Gamma$ of type 
 $(M,N)$, and has a canonical decomposition 
 \[ S_z = \Big( \bigcup_{i=1}^M \lceil_i a_i(z),b_i(z)\rfloor_i \Big) \cup \Big( \bigcup_{j=1}^N \{c_j(z)\} \Big) \]
 as in (1). Then the functions 
 $a_i,b_i,c_j : \SMN \to \Gamma$ are definable.
 \item 
 Let $S \subset K^m\times \Gamma^n \times \Gamma$ be a subanalytic set.
 There is only a finite number of $(M,N)\in \N^2$ such that 
 $\SMN \neq \emptyset$. In other words, only finitely many types are realized by 
 $\{S_z\}, \ z\in K^m\times \Gamma$, the family of definable sets of $ \Gamma$.  
\end{enumerate}
\end{fact}

The proof of our mixed cell decomposition theorem can now be done as in 
\cite[th 4.2.11]{VDD}:

\begin{theo}
\label{theomixedcelldec}
Let $S \subset K^m\times \Gamma^n$ be a definable set. 
There exists a partition 
\[S = \coprod_{i=1}^l C_i \]
where each $C_i$ is a cell of 
$K^m \times \Gamma^n$.
\end{theo}
\begin{proof}
We prove this by induction on $n$. \par 
For $n=0$, $S\subset \Kom$ is just a definable subset of 
$\Kom$, hence is a $S-\emptyset-$cell by definition. \par
Let $n\geq 0$, and $S\subset K^m \times \Gamma^n \times \Gamma$ be a definable set. 
According to lemma \ref{lemmebound}, there exists only a finite number of possible couples 
$(M,N) \in \N^2$, say $\mathcal{Q}$, such that $S^{(M,N)} \neq \emptyset$, i.e. such that for some 
$z \in K^m \times \Gamma^n$, $S_z$ has type $(M,N)$.
For such a couple 
$(M,N)$, we can then define functions $a_i,b_i,c_j : \SMN \to \Gamma$ which are definable 
(see lemma \ref{lemmebound} (4)). 
We then define 
\[]a_i^{(M,N)}, b_i^{(M,N)} [ \  := \big\{ (z,\gamma) \st z\in \SMN \ \and \ a_i(z)<\gamma<b_i(z) \}\]
\[ \big\{c_j^{(M,N)} \big\} \  := \Big\{ (z,\gamma) \st z\in \SMN \ 
\and \ c_j(z)= \gamma \Big\} .\]
By construction, $S$ is the disjoint union 
\[ S =   \bigcup_{(M,N) \in \mathcal{Q} }\Big( \bigcup_{i=1}^M ]a_i^{(M,N)},b_i^{(M,N)}[ \Big) \cup \Big( \bigcup_{j=1}^N \{c_j^{M,N} \} \Big) .\]
Now, according to proposition 
\ref{propfonccont}, we can shrink the definable sets 
$\SMN$ and assume that the $a_i,b_i,c_j's$ are continuous.
In addition, by induction hypothesis, decomposing each $S^{(M,N)}$ if necessary, we can assume that $\SMN$ 
is an $X-(i_1,\ldots,i_n)$-cell where 
$X \subset K^m$ is a definable set and $(i_1,\ldots,i_n) \in \{0,1\}^n$.
In that case, 
$]a_i^{(M,N)}, b_i^{(M,N)} [$ is a $X-(i_1,\ldots,i_n,1)$-cell and
$\{c_j^{(M,N)} \}$ is a $X-(i_1,\ldots,i_n,0)$-cell. 
\end{proof}

\subsection{Mixed dimension}
\label{sec2.2}

The mixed cell decomposition theorem will allow us to develop a theory of 
dimension for definable subsets of 
$\KmGn$ which extends the dimension theories for definable subsets of 
$K^m$ (as exposed in sections \ref{subsection:dimACVF} and \ref{sec1.2}).\par

\begin{defi}
\label{defmixeddim}
Let $S \subset \KmGn$ be a definable set. 
\index{dimension!of a definable set of $\Kom \times \Gamma^n$}
We define the dimension of $S$, denoted by $\dim(S)$, as the finite subset of 
$\llbracket 0,m \rrbracket \times \llbracket 0,n \rrbracket$:
\[\dim(S) = \{ (d_1,d_2) \in \N^2 \st S  \ \text{contains a cell C of dimension} \ (d_1,d_2) \}\]
where the dimension of a cell is the one given by definition \ref{defdimcell}.
\end{defi}
We want to point out that when $C$ is a cell, if $(d_1,d_2)$ is its dimension as defined in \ref{defdimcell}, then 
the dimension of $C$ according to \ref{defmixeddim} is $\langle (d_1,d_2)\rangle$ 
(this will be proved in \ref{propgen} (2) ). \par 
If $D$ is a finite subset of $\N^2$, we will represent it with one symbol $\bullet$ for each point of $D$.
For instance to the subset 
\[ D_1= \{(0,0),  (0,3), (0,4),(1,4),(2,0),(2,1),(2,2),(4,0),(4,1)  \}\]
will will associate the diagram:

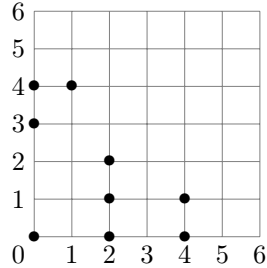
\begin{figure}[H]
\label{figure:D1}
\begin{tikzpicture}[scale=0.5]
\draw [very thin, gray] (0,0) grid (6,6);
\draw (0,0) node[below left] {0};
\foreach \x in {1,...,6} \draw (\x,0) node[below]{\x};
\foreach \y in {1,...,6} \draw (0,\y) node[left]{\y};
\draw (0,0) node {$\bullet$};
\draw (0,3) node {$\bullet$};
\draw (0,4) node {$\bullet$};
\draw (1,4) node {$\bullet$};
\draw (2,2) node {$\bullet$};
\draw (2,1) node {$\bullet$};
\draw (2,0) node {$\bullet$};
\draw (4,1) node {$\bullet$};
\draw (4,0) node {$\bullet$};
\end{tikzpicture}
\caption{Diagram associated to $D_1$.}
\end{figure}

In order to compare dimensions, we equip $\N^2$  with the partial order $\leq$ defined by 
$(d_1,d_2) \leq (d'_1,d'_2)$ if 
$d_1 \leq d'_1$ and $d_2 \leq d'_2$.
If $A$ and $B$ are subsets of $\N^2$, $A \leq B$ will mean $A\subset B$.\\par 
We will say that a set $D \subset \N^2$ is a lower set if for all 
$x\in \N^2 y\in D$, if $x\leq y$, then $x\in D$. 
In other words, if one represents $D$ in the plane, it is a lower set if it is stable in the directions $\leftarrow$ 
and $\downarrow$. For instance $D_1$ is not a lower set. 
If we denote by $D_2$ the smallest lower set which contains $D_1$, then $D_2$ is represented by the following diagram.

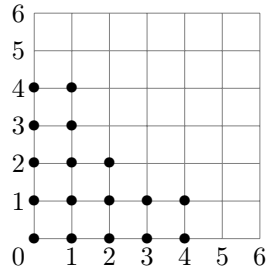
\begin{figure}[H]
\label{figure:D2}
\begin{tikzpicture}[scale=0.5]
\draw [very thin, gray] (0,0) grid (6,6);
\draw (0,0) node[below left] {0};
\foreach \x in {1,...,6} \draw (\x,0) node[below]{\x};
\foreach \y in {1,...,6} \draw (0,\y) node[left]{\y};
\foreach \y in {0,...,4} \draw (0,\y) node {$\bullet$};
\foreach \y in {0,...,4} \draw (1,\y) node {$\bullet$};
\foreach \y in {0,...,2} \draw (2,\y) node {$\bullet$};
\foreach \y in {0,...,1} \draw (3,\y) node {$\bullet$};
\foreach \y in {0,...,1} \draw (4,\y) node {$\bullet$};
\end{tikzpicture}
\caption{The diagram of $D_2$, the smallest lower set containing $D_1$.}
\end{figure}

A finite lower set of $\N^2$ marks out a bounded convex subset of $\R_+^2$ that we fill in with some small dots. 
With this graphical convention $D_2$ is represented by the diagram:

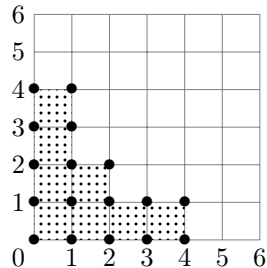
\begin{figure}[H]
\label{figure:2D2}
\begin{tikzpicture}[scale=0.5]
\draw [very thin, gray] (0,0) grid (6,6);
\draw (0,0) node[below left] {0};
\foreach \x in {1,...,6} \draw (\x,0) node[below]{\x};
\foreach \y in {1,...,6} \draw (0,\y) node[left]{\y};
\foreach \y in {0,...,4} \draw (0,\y) node {$\bullet$};
\foreach \y in {0,...,4} \draw (1,\y) node {$\bullet$};
\foreach \y in {0,...,2} \draw (2,\y) node {$\bullet$};
\foreach \y in {0,...,1} \draw (3,\y) node {$\bullet$};
\foreach \y in {0,...,1} \draw (4,\y) node {$\bullet$};
\fill [gray,pattern=dots] (0,0) -- (0,4) -- (1,4)--(1,2)--(2,2)--(2,1)--(4,1)--(4,0)--cycle;
\end{tikzpicture}
\caption{The diagram of $D_2$ which has been filled.}
\end{figure}
 
We will need to add dimensions. If $A$ and $B$ are subsets of $\N^2$, we will set 
\[A+B = \{a+b \st a\in A, \ b\in B\}.\]
More generally, if $A\subset \N^2$ and 
$B\subset \Z^2$, we will set 
\[A+ B = \{ a+b \in \N^2 \st a\in A, b\in B \}.\]
We will have to take maximum of dimensions. 
If $A,B$ are subsets of $\N^2$, we will set 
\[ \max (A,B) = A\cup B .\]
Note that if $A$ and $B$ are lower sets, $A\cup B$ is also a lower set.
\par  
Eventually, if $(d_1,d_2) \in \N^2$, we will set 
\[\langle (d_1,d_2) \rangle = \{ x\in \N^2 \st x\leq (d_1,d_2) \}.\]
This is the smallest lower set which contains $(d_1,d_2)$, and 
$(d_1,d_2)$ is the greatest element (w.r.t. $\leq$) of $\langle (d_1,d_2) \rangle$. 
We draw a picture to describe the situation if $(d_1,d_2) = (3,5)$:
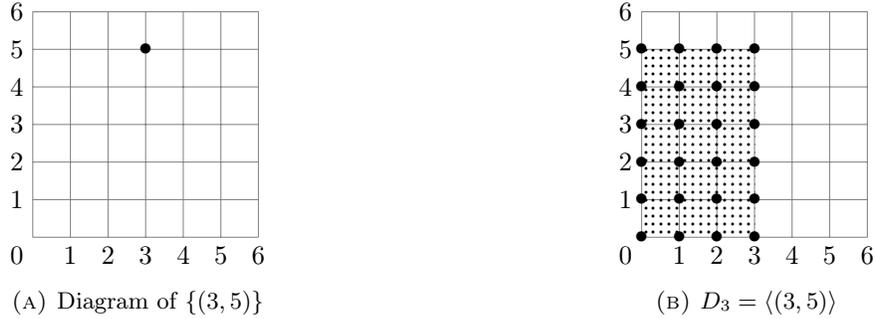
\begin{figure}[H]
\label{figure:D3}
\begin{center}

\begin{subfigure}[b]{0.4\textwidth}
\centering
\begin{tikzpicture}[scale=0.5]
\draw [very thin, gray] (0,0) grid (6,6);
\draw (0,0) node[below left] {0};
\foreach \x in {1,...,6} \draw (\x,0) node[below]{\x};
\foreach \y in {1,...,6} \draw (0,\y) node[left]{\y};
 \draw (3,5) node {$\bullet$};
\end{tikzpicture}
\caption{Diagram of $\{ (3,5)\}$}
\end{subfigure}
\hspace{2cm}
\begin{subfigure}[b]{0.4\textwidth}
\centering
\begin{tikzpicture}[scale=0.5]
\draw [very thin, gray] (0,0) grid (6,6);
\draw (0,0) node[below left] {0};
\foreach \x in {1,...,6} \draw (\x,0) node[below]{\x};
\foreach \y in {1,...,6} \draw (0,\y) node[left]{\y};
\foreach \x in {0,...,3}
\foreach \y in {0,...,5} \draw (\x,\y) node {$\bullet$};
\fill [gray,pattern=dots] (0,0) rectangle (3,5);
\end{tikzpicture}
\caption{$D_3 = \langle (3,5) \rangle$}
\end{subfigure}
\end{center}
\caption[Diagram of $D_3$]{On the left the diagram of $\{3,5\}$. On the right, $D_3:=\langle (3,5) \rangle$, 
the smallest lower set containing $(3,5)$.
For instance, if $S=\Kov{3} \times \Gamma^5$, then 
$\dim(S)=D_3$.}
\end{figure}
 
 \begin{lemme}
If $S$ is a definable set of $\KmGn$, $\dim(S)$ is a finite lower set of $\N^2$.
\end{lemme}
\begin{proof}
If $C$ is a $X-(i_1,\ldots,i_n)$-cell in $S$ with $d_1 = \dim(X)$ and $d_2=i_1+ \ldots+i_n$, 
and if 
$(d'_1,d'_2) \leq (d_1,d_2)$, we can choose 
$Y\subset X$ such that 
$\dim(Y) =d'_1$ and 
$C':= C \cap (Y \times \Gamma^n)$ is a 
$Y-(i_1,\ldots,i_n)$-cell of dimension $(d'_1,d_2)$. 
It is now easy to find a 
cell $C'' \subset C'$ of dimension $(d'_1,d'_2)$. We only have to decrease the $\Gamma$-dimension.
For instance, if $A$ is $X-(i_1,\ldots,i_n)$-cell, and 
$B=]f,g[$ a $X-(i_1,\ldots,i_n,1)$-cell above $A$, then 
$\Graph(\sqrt{fg})$ is a $X-(i_1,\ldots,i_n,0)$-cell included in $B$. 
\end{proof}

\begin{lemme}
\label{lemmeinjcell}
\begin{enumerate}
\item 
Let $C \subset \KmGn$ be a cell of dimension $(d_1,d_2)$, 
and assume that \[C= C_1 \cup \ldots \cup C_N \] where 
each $C_i$ is a cell. 
Then there exists some index $j \in 1\ldots N$ such 
that $(d_1,d_2) = \dim(C_j)$.
\item
If $f : C \hookrightarrow C'$ is an injective definable map between cells, 
$\dim(C) \leq \dim(C')$.
\end{enumerate}
\end{lemme}

\begin{proof}
\begin{enumerate}
\item 
Let $C$ be a $X-(i_1,\ldots,i_n)$-cell and $\{C_j\}_{j=1\ldots N}$ some 
$X_j-(i^j_1,\ldots,i^j_{n})$-cells such that 
$\displaystyle C=\cup_{j=1\ldots N} C_j$. 
Then for all $j$, it is easy to see that 
$X_j \subset X$, hence that 
$\dim(X_j) \leq \dim(X)$. \par 
In addition if $x \in C_j$, then 
$(C_j)_x:= \{ \gamma \in \Gamma^n \st (x,\gamma) \in C_j \}$ is a $\Gamma$-definable set included in $C_x$, this proves that 
$\dim(C_j) \leq d$. \par 
Now, $\max_{j=1\ldots N} (\dim(X_j) ) =\dim(X)$ (see \ref{propdimunionsub}). 
Let then 
\[ X' := X \setminus ( \bigcup_{\underset{\dim(X_i) < \dim(X) }{i \ \text{such that} } } X_i ). \]
Then $\dim(X') = \dim(X)$ and replacing $X$ by $X'$, we can assume that 
$\dim(X_j) = \dim(X)$ for all $j$. \par
Hence if $x\in X$, 
$C_x = \cup_j(C_j)_x$, hence for some $j$, 
$(C_j)_x$ has the same dimension as $C_x$. So $C_j$ is a cell of the same dimension as $C$.
\item 
Let $C$ be a $X-(i_1,\ldots,i_l)$-cell and 
$C'$ a $Y-(j_1,\ldots,j_n)$-cell and 
$f: C \hookrightarrow C'$ a definable injective map.  
We first remark\footnote{ This can be proved by 
induction on $l$. For instance if  $C = ]f,g[$ is a $X-(1)$-cell, we take 
\[
\begin{array}{cccc}
s : &X &\hookrightarrow &C\\
  &  x & \mapsto & (x, \sqrt{f(x)g(x)})
  \end{array}.
  \] 
} 
that the first projection 
$\pi : C \to X$ has a definable section 
$s : X \hookrightarrow C$.
We then obtain an injective map 
$g : X \hookrightarrow C'$. \par  
For $i=1 \ldots n$, let $g_i$ be the composite of $g$ with the i-th projection on $\Gamma$.
According to \ref{propfonccont}, we can even assume that 
$g_i = \sqrt[a_i]{ |G_i| }$ for some continuous 
definable function $G_i : X \to K^*$.
Is is then easy to see that that we can shrink $X$ into $V$
and assume that $\dim(V)=\dim(X)$ and that the $(g_i)_{|V}$'s
are constant.\par
Hence if we project on $Y$ we now obtain a 
definable injection 
$V \hookrightarrow Y$. Hence 
$\dim(Y) \geq \dim(V) = \dim(X)$. 
This proves that 
$d_1 \leq d'_1$. \par
Now, let $x\in X$. 
Then $f$ induces a map 
$f : C_x \to C' \subset Y \times \Gamma^n$ which is not necessarily 
injective. 
However, if we project on $Y$, we obtain a map 
$\alpha : C_x \to Y$ with $C_x \subset \Gamma^l$, and $Y \subset K^p$ for some $p$. 
Hence according to lemma \ref{finite_image}, $\alpha (C_x)$ is finite.
Hence there exists a definable subset 
$P \subset C_x$ such that 
$\alpha_{|P} : P \to Y$ is constant, with image $y$ say, and such that $\dim(P) = \dim(C_x) =d_2$. 
Now, 
$f_{|P} : P \to C'$ is injective and even induces an injective map 
$P \hookrightarrow (C')_y$ which is a $\Gamma$-definable set of dimension 
$d'_2$. Hence $d_2 = \dim(P) \leq d'_2$.

\end{enumerate}
\end{proof}

We now give some general properties satisfied by dimension of definable subsets of 
$\KmGn$. 

\begin{prop}
\label{propgen}
$ $
\begin{enumerate}
\item
Let $X \subset Y \subset \KmGn$ be definable subsets. Then 
$\dim(X) \leq \dim (Y)$. 
\item
Let $C \subset \KmGn$ be a cell of dimension $(d_1,d_2)$ in the sense of definition \ref{defdimcell}. 
Then, in the sense of definition , $\dim(C) = \langle (d_1,d_2) \rangle$.
\item 
If $f : X \hookrightarrow Y$ is a definable injection of definable subsets of 
$\KmGn$ (resp. $K^{m'} \times \Gamma^{n'}$), then $\dim(X) \leq \dim(Y)$. \par
If $f$ is bijective, $\dim(X)= \dim(Y)$.
\item 
$\dim(X \cup Y) = \max (\dim(X),\dim(Y) )$.
\item 
If $X = \cup_{i=1}^N C_i$ is a union of cells of dimension $(d_1^i, d_2^i)$ (according to definition 
\ref{defdimcell}), then 
$\dim(X) = \max_i \langle (d^i_1,d^i_2) \rangle$.
\end{enumerate}
\end{prop}

\begin{proof}$ $
\begin{enumerate}
\item 
Is clear.
\item
Is a consequence of lemma \ref{lemmeinjcell} (2).
\item 
For the first assertion, let $C\subset X$ be a cell of dimension $(d_1,d_2)$. 
Then according to the cell decomposition theorem, we can find a decomposition 
$f(C) = C'_1 \cup \ldots \cup C'_N$ in $N$ cells. 
Then if we apply the cell decomposition theorem \ref{theomixedcelldec} to 
all the $f^{-1}(C'_i)$, we can find a decomposition in cells
$C = C_1 \cup \ldots \cup C_M$ such that
for all $i$, there exists some $j$ such that 
$f(C_i) \subset C'_j$. \par 
Now, according to  
lemma \ref{lemmeinjcell} (1), there exists some $i$ 
such that $\dim(C_i) = \dim(C)$.
Since $f_{|C_i} : C_i \hookrightarrow C'_j$ is injective, according to 
lemma \ref{lemmeinjcell} (2), 
$\dim(C_i) \leq \dim(C'_i) \leq \dim(Y)$. 
Hence this proves that 
$\dim(X) \leq \dim(Y)$.\par
To prove the second assertion, we apply the first one to $f$ and $f^{-1}$.
\item 
Clearly, 
$\max(\dim(X), \dim(Y) ) \leq \dim(X\cup Y)$. 
Conversely, let $C\subset X\cup Y$ be a cell. 
We apply the cell decomposition theorem to 
$X \cap C$ and $Y\cap C$:
\[ X \cap C = C_1\cup \ldots \cup C_M, \ \and \ Y\cap C = C'_1\cup \ldots \cup C'_N.\]
Then, $C= C_1\cup \ldots \cup C_M \cup C'_1\cup \ldots \cup C'_N$, hence according to lemma \ref{lemmeinjcell} (2), 
for some $i$ 
$\dim(C_i) = \dim(C)$ or for some $j$, $\dim(C'_j) = \dim(C)$, hence 
$\dim(C) \leq \max (\dim(X), \dim(Y))$.
\item Is a consequence of (2) and (4).
\end{enumerate}
\end{proof}

\subsection{Behavior under projection}
\label{sec2.3}

\begin{lemme}
\label{lemmefondim1}
Let $K$ be an algebraically closed valued field. 
Let $f_1,\ldots,f_N \in K[T]$. 
There exist
\begin{itemize}
 \item a decomposition 
$K = V_1\cup \ldots \cup V_M$ of 
$K$ in definable subsets\footnote{which are 
in fact semialgebraic according to \cite[4.7]{LRline}.} 
\item an integer $m$
\item for each $j=1 \ldots m$, some
$a_j \in K$
\item for all $1\leq i \leq N$, $1\leq j \leq M$, 
some integers $n_{i,j} \in \N$ and some $c_{i,j} \in \Gamma$ 
\end{itemize} 
such that for all $1\leq i \leq N$, $1\leq j \leq M$,
\[
 |f_i(x)| = c_{i,j} |x-a_j|^{n_{i,j}}.\]
\end{lemme}
We will not write the proof, which consists in a precise analysis of the combinatorics of some closed balls 
associated with the configuration of the roots of the $f_i$'s. 
In the non-Archimedean case, the Berkovich line $\mathbb{A}^{1, \an}_K$ gives a good insight on it.

\begin{lemme}
\label{lemmedimproj}
Let $C \subset K\times \Gamma^n$ be a cell of dimension 
$(1,d)$. 
Let $\pi : K \times \Gamma^n \to \Gamma^n$ be the second projection. Then $\dim( \pi(C) ) \leq d+1.$.
\end{lemme}
\begin{proof}
Assume that $C$ is a 
$X-(\icell)$-cell, where $X$ is a definable subset of $K$  with $\dim(X)=1$, and 
$i_1+\ldots+i_n=d$.
Then by definition of a cell, there exists a sequence 
$C_0, \ldots C_n$ such that 
\begin{itemize}
\item 
$C_0 = X$.
\item 
For each $k=0 \ldots n$, $C_k$ is a 
$X-(i_1,\ldots,i_k)$-cell.
\item For each 
$k=0 \ldots n-1$ 

\begin{itemize}
\item if $i_{k+1} = 0$, there exists a (continuous) definable map 
$f_k : C_k \to \Gamma$ such that 
$C_{k+1} = \Graph(f_k)$. 
\item If $i_{k+1}=1$, there exist (continuous) definable maps
$f_k, g_k : C_k \to \Gamma$ such that $f_k<g_k$ and 
$C_{k+1} = ]f_k, g_k[$.
\end{itemize}

\end{itemize}
Remind that by construction, $C_k$ is a definable subset of 
$X\times \Gamma^k \subset K \times \Gamma^k$, hence 
$f_k, g_k : X \times \Gamma^k \to \Gamma$ are definable maps. 
According to proposition \ref{propfonccont}, we can shrink $X$ such that 
\[ 
\begin{array}{cccc}
f_k :& X \times \Gamma^k & \to & \Gamma \\
     & (x,\gamma) &\mapsto & \sqrt[d_k]{|F_k(x)| \gamma^{u_k} }
\end{array}
\]
for some definable $F_k : X \to K$, some 
$u_k\in \Z^k, d_k\in \N^*$, and similarly,  
\[ g_k(x,\gamma) = \sqrt[e_k]{ |G_k(x)| \gamma^{v_k} }.\]
Now, according to \cite[3.3]{LRline}, we can again shrink 
$X$ and assume that 
$F_k, G_k \in K(T)$ are fractions without pole on $X$. \par
Now, according to lemma \ref{lemmefondim1}, we can again 
shrink $X$ and assume that \\
there exist some $a\in K, \ 
m_k,n_k \in \Z, b_k,c_k\in \Gamma$ such that for all $x\in X$
\begin{itemize}
\item 
$|F_k(x)| = b_k|x-a|^{m_k}$ 
\item 
$|G_k(x)| = c_k |x-a|^{n_k}$.
\end{itemize}
We can again shrink $X$ so that 
$I = \{|x-a| \st x\in X \}$ is a cell of $\Gamma$ (that is to say, a singleton 
or an open interval). 
Then if we consider the map 
\[
\begin{array}{cccc}
\varphi : & X\times \Gamma^n & \to & \Gamma \times \Gamma^n \\
           & (x,\gamma) & \mapsto & (|x-a| , \gamma)

\end{array} 
\]
it is easily checked that 
$C' := \varphi(C) \subset \Gamma\times \Gamma^n$ is a cell: 
\begin{itemize}
\item 
A $(1,i_1,\ldots, i_n)$-cell if $I$ is an interval, hence of dimension $d+1$.
\item A $(0,i_1,\ldots,i_n)$-cell if $I$ is a singleton, hence of dimension $d$.
\end{itemize}
Now if 
$p : \Gamma \times \Gamma^n \to \Gamma^n$ is the coordinate projection along 
the last $n$ coordinates, by construction,  $\pi = p \circ \varphi$.
 Hence  
$\pi(C) = p (\varphi(C)) =p(C')$ and since the possible dimensions of $C'$ 
are $d,d+1$, the possible dimensions for $\pi(C)$ are also 
$d,d+1$ (see \ref{propimagedim}).
\end{proof}

\begin{prop}
\label{propimage}
Let $f:X \to Y$ be a definable map of definable subsets,  
$X \subset \KmGn$, $Y\subset K^{m'} \times \Gamma^{n'}$. Then 
\[ \dim(f(X)) \leq \max_{d\geq 0} (\dim(X) + (-d,d) ).\]
\end{prop}
Before giving the proof, let us explain what this mean on a diagram. 
If $\dim(X)=D_2$ for instance, then the possibilities for 
$\dim(f(X))$ are encoded by the digram obtain by the digram of $D_2$ that we extend by adding new points, following arrows $\searrow$ in the south-east direction. 
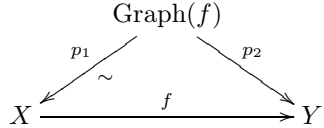
\begin{figure}[H]
\label{figure:D4}
\begin{center}
\begin{subfigure}[b]{0.3\linewidth}
\centering
\begin{tikzpicture}[scale=0.5]
\draw [very thin, gray] (0,0) grid (7,7);
\draw (0,0) node[below left] {0};
\foreach \x in {1,...,7} \draw (\x,0) node[below]{\x};
\foreach \y in {1,...,7} \draw (0,\y) node[left]{\y};
\foreach \y in {0,...,4} \draw (0,\y) node {$\bullet$};
\foreach \y in {0,...,4} \draw (1,\y) node {$\bullet$};
\foreach \y in {0,...,1} \draw (2,\y) node {$\bullet$};
\foreach \y in {0,...,1} \draw (3,\y) node {$\bullet$};
\foreach \y in {0,...,1} \draw (4,\y) node {$\bullet$};
\foreach \y in {0,...,1} \draw (5,\y) node {$\bullet$};
\fill[gray,pattern=dots] (0,0) -- (0,4) -- (1,4)--(1,1)--(5,1)--(5,0)--cycle;
\end{tikzpicture}
\caption{Diagram of $D_4$}
\end{subfigure}
~
\begin{subfigure}[b]{0.3\linewidth}
\centering
\begin{tikzpicture}[scale=0.5]
\draw [very thin, gray] (0,0) grid (7,7);
\draw (0,0) node[below left] {0};
\foreach \x in {1,...,7} \draw (\x,0) node[below]{\x};
\foreach \y in {1,...,7} \draw (0,\y) node[left]{\y};
\foreach \y in {0,...,4} \draw (0,\y) node {$\bullet$};
\foreach \y in {0,...,4} \draw (1,\y) node {$\bullet$};
\foreach \y in {0,...,1} \draw (2,\y) node {$\bullet$};
\foreach \y in {0,...,1} \draw (3,\y) node {$\bullet$};
\foreach \y in {0,...,1} \draw (4,\y) node {$\bullet$};
\foreach \y in {0,...,1} \draw (5,\y) node {$\bullet$};
\foreach \x in {1,...,3} \draw [directed] (\x, 5-\x) -- (\x+1,4-\x);
\foreach \x in {0} \draw[directed] (\x+5, 1-\x) -- (\x+6,0-\x);
\fill[gray,pattern=dots] (0,0) -- (0,4) -- (1,4)--(1,1)--(5,1)--(5,0)--cycle;
\end{tikzpicture}
\caption{in between}
\end{subfigure}
~
\begin{subfigure}[b]{0.35\linewidth}
\centering
\begin{tikzpicture}[scale=0.5]
\draw [very thin, gray] (0,0) grid (7,7);
\draw (0,0) node[below left] {0};
\foreach \x in {1,...,7} \draw (\x,0) node[below]{\x};
\foreach \y in {1,...,7} \draw (0,\y) node[left]{\y};
\foreach \y in {0,...,4} \draw (0,\y) node {$\bullet$};
\foreach \y in {0,...,4} \draw (1,\y) node {$\bullet$};
\foreach \y in {0,...,3} \draw (2,\y) node {$\bullet$};
\foreach \y in {0,...,2} \draw (3,\y) node {$\bullet$};
\foreach \y in {0,...,1} \draw (4,\y) node {$\bullet$};
\foreach \y in {0,...,1} \draw (5,\y) node {$\bullet$};
\draw (6,0) node {$\bullet$};
\fill[gray,pattern=dots] (0,0) -- (0,4) -- (1,4)--(1,3)--(2,3)--(2,2)--(3,2)--(3,1)--(5,1)--(5,0)--cycle;
\end{tikzpicture}
\caption{$ D_5=\max_{d\geq 0}( D_4 + (-d,d) )$}
\end{subfigure}
\caption[Diagrams of $D_4$ and $D_5$]{On the left, an example of a finite lower set $D_4$.
For instance, if $S=\Ko \times \Gamma^4 \cup \Kov{5} \times \Gamma \subset \Kov{6} \times \Gamma^4$, then $\dim(S)=D_4$.  
In the middle, we explain how we can obtain new dimensions for $\dim(f(X))$. 
On the right we obtain the finite 
lower set $\displaystyle D_5 = \max_{d\geq 0} (D_4 + (-d,d) )$. 
The above proposition asserts that $\dim(f(X)) \leq D_5$.}
\end{center}
\end{figure}
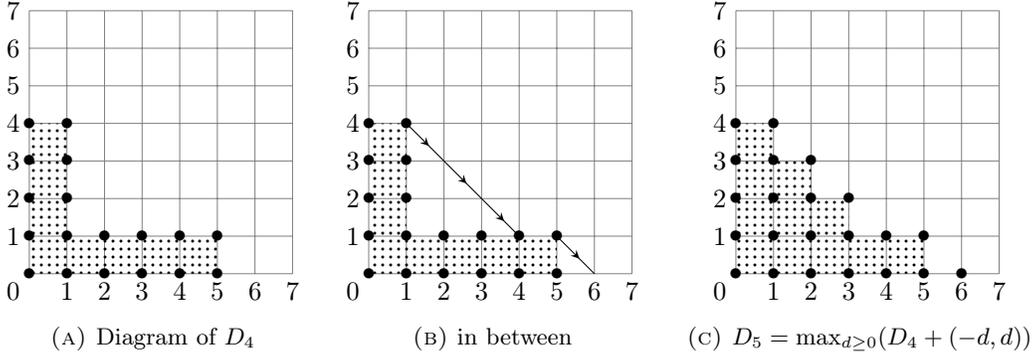
\begin{proof}
Consider 
\[ 
\xymatrix{
 & \Graph(f) \ar[dl]^{\sim}_{p_1} \ar[dr]^{p_2}& \\
 X \ar[rr]^f & &Y 
}
\]
where $\Graph(f)$ is a definable set of 
$K^{m+m'} \times \Gamma^{n+n'}$. \par 
According to proposition \ref{propgen} (3), 
$\dim(X) = \dim(\Graph(f))$. 
Hence we can assume that 
$f =p_2$, that is to say, we can assume that $f$ is a coordinate projection.\par
We can even assume that $X$ is a cell, say a cell $C$ of dimension 
$(d_1,d_2)$. 
Let us assume that 
$C$ is a $T-(i_1,\ldots,i_n)$-cell. \par
Making an induction, we can even assume that $f$ is a projection 
along a single coordinate, because $(d,-d) + (1,-1) = (d+1,-(d+1))$.\par
We then distinguish the two cases corresponding to a projection along 
a variable of $K$ or of $\Gamma$.
\paragraph{Case 1}
Let us assume that  $f$ is the projection 
\[ f : K^m \times K \times \Gamma^n \to K^m \times \Gamma^n. \]
Let us set $S := \pi (T)$ where 
\[ \pi : K^m \times K  \to K^m \]
is the projection along the first $m$-coordinates. \par 
Then after a partition of $S$, we can assume that 
$S = T(0)$ or $S=T(1)$ with respect to the projection $\pi$ (see \ref{defS(i)sub}).
Remind that $T(i)$ is the set of points $s$ of $S$ such that 
$\dim(\pi^{-1}(s) \cap T)=i$. We distinguish these two cases. 
\paragraph{Case 1.1}
Let us assume that $S=T(0)$. Let then $D \subset f(C)$ be a cell, say an 
$U-(j_1,\ldots,j_n)$-cell. Then by definition $U \subset S$ hence 
$\dim(U) \leq \dim(T)$, and for 
$u \in U$, $\pi^{-1}(u)\cap T$ is finite. 
This implies that 
$D_u$ is the projection along $f$ of a finite number of 
$(i_1,\ldots,i_n)$-cells. Hence 
$j_1+ \ldots +j_n= \dim(D_u) \leq i_1+ \ldots + i_n$.
Hence in that case $\dim(D) \leq \dim (C)$. 
\paragraph{Case 1.2}
Let us assume that $S=T(1)$. 
Then according to 
proposition \ref{defS(i)sub}, 
\[\dim(S) = \dim(T)-1.\] 
Since $S =\pi(T)$ by definition, it follows that   
$f(C) \subset S \times \Gamma^n$.\par  
Let then $D \subset \pi(C)$ be cell. Say $D$ is a 
$U-(j_1,\ldots,j_n)$-cell. 
If we prove that $\dim(D) \leq \max (\dim(C), \dim(C)+(-1,1) )$, this will conclude the proof. 
Necessarily, $U \subset S$, hence 
\[\dim(U) \leq \dim(S) \leq d_1-1.\] 
Now, if $u\in U$, then 
\[ D_u := \{ \gamma \in \Gamma^n \st (u,\gamma) \in D \} \] 
is a 
$(j_1,\ldots,j_n)$-cell and 
$D_u \subset p(C_u)$ where 
\[p : K \times \Gamma^n \to \Gamma^n \]
and
\[C_u :=  \{ (x,\gamma) \in K \times \Gamma^n \st (u,x,\gamma) \in C \}.\]
Now $C_u$ can be decomposed in some cells, whose dimension are 
$(0,d)$ or $(1,d)$ with 
$d \leq d_2$. \par 
Now the projection of a $(0,d)$-cell along $p$ is clearly a 
$d$-cell of $\Gamma^n$. \par 
According to lemma \ref{lemmedimproj}, the projection of a 
cell of dimension 
$(1,d)$ along $p$ is a cell of $\Gamma^n$ of dimension $\leq d+1$. 
Hence, $\dim(D_u) \leq d+1 \leq d_2 +1$, which means that 
$D$ is a cell of dimension 
\[ \dim(D)\leq (d_1-1,d_2+1) \leq \max( \dim(X), \dim(X)+ (-1,1) ).\]
This ends the case 1 of a projection along $K$. 
\paragraph{Case 2}
Let us assume that $f$ is a projection 
\[ f: \KmGn \times \Gamma \to \KmGn.\]
It is then clear that if 
$C \subset K^{m} \times \Gamma^{n+1}$, is a 
$T-(i_1, \ldots, i_n, i_{n+1} )$-cell, then 
$f(C)$ is a 
$T-(i_1,\ldots,i_n)$-cell. 
So 
$\dim(f(C)) \leq \dim(C)$.
\end{proof}

We are now able to prove the theorem mentioned in the introduction:
\begin{theo}
\label{theodimmixed}
The mixed dimension for definable sets of $\KmGn$ satisfies the following properties:
\begin{enumerate}
\item 
If $S \subset \Gamma^n$, then 
$\dim(S) = \langle (0,d) \rangle$ where $d$ is the classical dimension of 
$S$ as a subset of $\Gamma^n$ as defined in section \ref{sec1.4}. 
\item If $S \subset K^m$ is definable, then 
$\dim(S) = \langle (d,0) \rangle$ where $\dim$ is the classical dimension of 
definable subsets of $K^m$ (as defined in sections \ref{subsection:dimACVF} and \ref{sec1.2}).
\item 
If $f : S \to T$ is a definable map, where 
$S \subset \KmGn$ and $T\subset K^{m'}\times \Gamma^{n'}$ are definable, then
\[ \dim (f(S) ) \leq \max_{k\geq 0} ( \dim(S) + (-k,k) ). \]
\item 
If $f: S \to T$ is a definable bijection, then $\dim(S)=\dim(T)$. 
\item $\dim(S \times T) = \dim(S)+ \dim(T)$.
\end{enumerate}
\end{theo}
\begin{proof}
The main difficulty is to prove $(3)$, which results from proposition \ref{propimage}.
\end{proof}

\begin{rem}
\label{remsimple}
If $X \subset \KmGn$ is definable, for each 
integer $i$, 
\[X(i) := \{ x\in K^m \st X_x \ \text{is a definable set of} \  \Gamma^n 
\ \text{of dimension} \ i \} \]
is easily seen to be a definable set of $K^m$.
Then the mixed dimension that we have build can also be characterized by:
\[ \dim(X) = \max_{ \underset{i \geq 0}{X_i \neq \emptyset}} \langle ( \dim(X(i)), i) \rangle.\]
\end{rem}

\subsection{$\N$-valued dimension theory}
\label{sec:dimN}
As a corollary of the previous theorem we obtain a 
$\N$-valued dimension theory for definable sets of $\KmGn$ which satisfies all expected properties. 
\begin{defi}
Let $S \subset  \KmGn$ be a definable subset. 
We define 
\[\dim_\N := \max_{(d_1,d_2) \in \dim(X) } d_1+d_2.\]
\end{defi}
Graphically, it is easy to compute $\dim_\N(X)$ form 
$\dim(X)$. The following diagram should be convincing.
\begin{figure}[H]
\label{figure:Ndim}
\begin{center}
\begin{tikzpicture}[scale=0.5]
\draw [very thin, gray] (0,0) grid (7,7);
\draw (0,0) node[below left] {0};
\foreach \x in {1,...,7} \draw (\x,0) node[below]{\x};
\foreach \y in {1,...,7} \draw (0,\y) node[left]{\y};
\foreach \y in {0,...,4} \draw (0,\y) node {$\bullet$};
\foreach \y in {0,...,4} \draw (1,\y) node {$\bullet$};
\foreach \y in {0,...,2} \draw (2,\y) node {$\bullet$};
\foreach \y in {0,...,1} \draw (3,\y) node {$\bullet$};
\foreach \y in {0,...,1} \draw (4,\y) node {$\bullet$};
\draw [dashed] (0,5)--(5,0);
\fill [gray,pattern=dots] (0,0) -- (0,4) -- (1,4)--(1,2)--(2,2)--(2,1)--(4,1)--(4,0)--cycle;
\end{tikzpicture}
\hspace{2cm}
\begin{tikzpicture}[scale=0.5]
\draw [very thin, gray] (0,0) grid (7,7);
\draw (0,0) node[below left] {0};
\foreach \x in {1,...,7} \draw (\x,0) node[below]{\x};
\foreach \y in {1,...,7} \draw (0,\y) node[left]{\y};
\foreach \y in {0,...,4} \draw (0,\y) node {$\bullet$};
\foreach \y in {0,...,4} \draw (1,\y) node {$\bullet$};
\foreach \y in {0,...,1} \draw (2,\y) node {$\bullet$};
\foreach \y in {0,...,1} \draw (3,\y) node {$\bullet$};
\foreach \y in {0,...,1} \draw (4,\y) node {$\bullet$};
\foreach \y in {0,...,1} \draw (5,\y) node {$\bullet$};
\draw [dashed] (0,6)--(6,0);
\fill[gray,pattern=dots] (0,0) -- (0,4) -- (1,4)--(1,1)--(5,1)--(5,0)--cycle;
\end{tikzpicture}
\caption[How to calculate $\dim_\N$. Example with $D_2$ and $D_4$.]{We consider the smallest integer $d$ such that the dashed line $x+y=d$ is above the lower set $D=\dim(X)$. On the left for $D_2$ we obtain $d=5$, and on the right, for $D_3$ we obtain $d=6$.}
\end{center}
\end{figure}
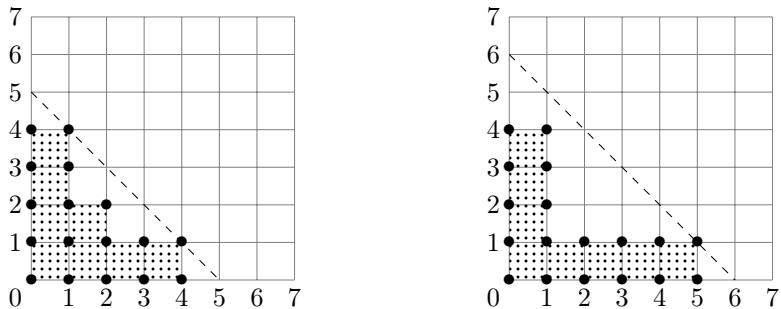

Property (iii) below proceeds from the fact that 
if $(d_1,d_2) \in \N^2$ and $i\in \N$, 
$(d_1-i) +(d_2+i) \leq d_1+d_2$.
\begin{theo}
\label{theodimmN}
\begin{enumerate}
\item 
If $S \subset K^m$ is definable, $\dim_\N(S)$ is the \emph{classical} dimension of 
$S$ (as defined in sections \ref{subsection:dimACVF} and \ref{sec1.2}).
\item 
If $S \subset \Gamma^n$ is definable, $\dim_\N(S)$ is the \emph{classical} dimension of 
$S$ (as defined in section \ref{sec1.4}).
\item 
If $f : S \to T$ is a definable map where $S$ (resp. $T$) is 
a definable set of 
$\KmGn$ (resp. $K^{m'} \times \Gamma^{n'}$), then
$\dim_\N (f(S) ) \leq \dim_\N(S)$.
\item 
If $f: S \to T$ is a definable bijection, then $\dim_\N(S)=\dim_\N(T)$. 
\item $\dim_\N(S \times T) = \dim_\N(S)+ \dim_\N(T)$.
\end{enumerate}
\end{theo}

\section{Dimension with three sorts: generic points}
\label{section3sortes}
In this section, $K$ will be either an algebraically closed valued field, or 
a non-Archimedean algebraically closed field. In the second case, 
we will assume that 
$(\R^*_+ : \Gamma) = +\infty$. 
This technical hypothesis can however  be dropped. \par 
According to the setting (ACVF or subanalytic), the notion of 
a definable set 
$S \subset \KmGn\times k^p$ has to be understound 
accordingly to sections \ref{subsubsec:ACVF} and \ref{subsubsec:an}.

\subsection{Subsets of $\KmGn$, link with the previous section}
\label{sec:2sortsgen}
\begin{prop}
Let $S \subset \KmGn$ be a definable set, and $(a,b) \in \N^2$. 
The following statements are equivalent:
\begin{enumerate}
\item $(a,b) \in \dim(S)$.
\item There exists an algebraically closed (non-Archimedean in the analytic case) extension $K \to L$ and 
$(x, \gamma) \in S(L)$ such that 
\[(a,b) = (d(K)(x)/K) , \dim_{\Q}( |K(x)^*|\langle \gamma \rangle / |K(x)^*| ).\]
\end{enumerate}
\end{prop}
 \begin{proof}
 $(1) \Rightarrow (2)$ Using a cell decomposition of $S$, 
we can assume that $S$ is a cell of dimension $(a,b)$.
Say $S$ is a $X-(i_1, \ldots,i_n)$-cell. 
Then there exists an algebraically closed non-Archimedean extension $K \to L$ 
and $x \in X(L)$ such that $d(K(x) /K)=a$. 
Consequently, $S_x \subset |L^*|^n$ is an 
$(i_1,\ldots,i_n)$-cell defined over 
$|K(x)^*|$. 
Therefore there exists $\gamma \in S_x(\R_+^*)$ such that 
\[\dim_{\Q} \Big( ( |K(x)^*\langle \gamma \rangle| :  |K(x)^*| ) \OTQ\Big)=b.\]
We can then increase $L$ if necessary such that 
$\gamma \in |L^*|$ and then 
$(x,\gamma) \in S(L)$ and 
\[ (a,b) = (d(K)(x)/K) , \dim_{\Q}( |K(x)^*|\langle \gamma \rangle |: |K(x)^*| ).\]
$(2) \Rightarrow (1)$
We can again assume that $S$ is some $X-(i_1,\ldots,i_n)$-cell.
Then 
\[a = d(K(x) /K) \leq \dim(X) \]
and since $S_x$ is an $(i_1,\ldots,i_n)$-cell defined over $|K(x)^*|$, 
\[ b = \dim_{\Q} \Big( (|K(x)^*|\langle \gamma \rangle) \OTQ \Q \Big) \leq i_1+\ldots i_n.\]
Hence 
$(a,b) \leq (\dim(X), i_1 + \ldots +i_n)$ which is equivalent to say that 
$(a,b) \in \dim(S)$.
 \end{proof}

 \subsection{subsets of $\KmGn \times k^p$}
\begin{defi}
\label{defi:d3sortes}
Let $K \to L$ be an extension and 
$(x,\gamma,\lambda) \in L^m \times \Gamma_L^n \times k_L^p$. We set 
\[d( (x,\gamma,\lambda)/K) :=
\Big( \  d(K(x)/K) , \dim_{\Q} \big( ( |K(x)^*|\langle \gamma \rangle| / |K(x)^*|) \OTQ \big) , \td( \widetilde{ K(x)}( \lambda) / \widetilde{K(x)} ) \ \Big).\]
\end{defi}

\begin{defi}
Let $X\subset K^m\times \Gamma^n \times k^p$ be some definable set. We define 
\[ \dim (X) := \{ (d(x,\gamma,\lambda) \st (x,\gamma,\lambda) \in X(L)  \ \text{for}  \ K\to L \ \text{some extension} \}.
\]
\end{defi}
Beware that this is not a correct definition, because the class of all extensions $K \to L$ is not a set. 
To obtain a correct definition, one can for instance impose a bound on the cardinal of $L$.
In the ACVF case, one way to obtain a correct definition is to consider one $\Card(L)^+$-saturated model $L$ and only consider 
the points $x\in X(L)$ (see \cite[section 4.3]{Mark} for the notion of saturated model). \par 
Hence, by definition, $\dim(X)$ is a finite subset of  
$\llbracket 0 \ldots m \rrbracket \times \llbracket 0 \ldots m \rrbracket \times \llbracket 0 \ldots p \rrbracket$. 
Moreover, one can check by induction that $\dim(X)$ is stable under $\leq$ where $\leq$ is the partial ordering defined component wise on $\N^3$.

\begin{lemme}
Let $X_1 \subset K^m\times \Gamma^n \times k^p$,  
$X_2 \subset  K^{m'}\times \Gamma^{n'} \times k^{p'}$ be some 
definable sets.
Let $f : X_1 \to X_2$ be a definable bijection.
Let $K \to L$ be some algebraically closed valued extension and $(x_1,\gamma_1,\lambda_1) \in X_1(L)$ and 
$f(z)= (x_2,\gamma_2,\lambda_2) \in X_2(L)$. 
Then $d((x_1,\gamma_1,\lambda_1)/K) = d((x_2,\gamma_2,\lambda_2)/K)$.
\end{lemme}

\begin{proof}
Let $M$ be the completion of the algebraic closure of 
$K(x_1)$ inside $L$, so that we have field inclusions
$K \to M \to L$. 
Let then 
\[Y:= \{ (x,\gamma,\lambda) \in X(L) \st x=x_1 \}.\]
This is naturally isomorphic to a definable set of $\Gamma_L \times k_L$ (with parameters in $L$).
Hence the composition 
\[ Y(L) \xrightarrow[]{f_{|Y}} X_2(L) \to L^{m'} \]
has finite image according to lemma \ref{finite_image}, where the second arrow is the projection 
$L^{m'} \times \Gamma^{n'} \times k^{p'} \to L^{m'}$. \par 
Then $f(Y(L)) = f(Y(M))$ (cf Remark \ref{rem:ens_fini}) so 
$x_2 \in M^m$ and by a symmetric argument it follows that the completion of the algebraic closure of 
$K(x_1)$ and $K(x_2)$ are the same. \par 
This allows to reduce to the case where $X_1 \subset \Gamma^n \times k^p$ and 
$X_2 \subset \Gamma^{n'} \times k^{p'}$. 
Let us then consider 
$(\gamma_1,\lambda_1) \in X_1(L)$, and $(\gamma_2,\lambda_2) \in X_2(L)$ its image by $f$.
We can find some subfield 
$K \subset M \subset L$ such that $M$ is complete and algebraically closed and  
\begin{flalign*}
& \tilde{M} = k(\lambda_1)^\alg & \\
& | M^*| = |K^*|\langle \gamma_1 \rangle \OTQ
\end{flalign*}
Then $(\gamma_2,\lambda_2) \in X_2(M)$ and 
\[ d(( \gamma_2,\lambda_2) / K) \leq ( \dim_{\Q} (|M^*| / |L^*| ), \td (\tilde{M}/ \tilde{K}) ) = d(\gamma_1,\lambda_1) .\]
By symmetry, $d((\gamma_1,\lambda_1)/K) = d((\gamma_2,\lambda2)/K)$.
\end{proof}
This has the following corollary: 

\begin{cor}
Let $f : X_1 \to X_2$ be a definable bijection. 
Then $\dim(X_1) = \dim(X_2)$.
\end{cor}

In order to understand what happens to the dimension under a definable map $f$, 
we can restrict to the case where $f$ is a projection since we have just proved that the dimension 
is preserved by definable bijections.
Hence let us consider $K \to L$ an algebraically closed extension and 
\begin{flalign*}
& \pi_1: L^{m+1} \times \Gamma^n \times k^p \to L^m \times \gamma^n \times k^p &\\
& \pi_2: L^{m} \times \Gamma^{n+1} \times k^p \to L^m \times \gamma^n \times k^p &\\
& \pi_1: L^{m} \times \Gamma^n \times k^{p+1} \to L^m \times \gamma^n \times k^p &
\end{flalign*}

\begin{theo}
Let $X_1 \subset K^m\times \Gamma^n \times k^p$, 
$X_2 \subset K^{m'}\times \Gamma^{n'} \times k^{p'}$ be definable sets and 
$f : X_1 \to X_2$ be a definable map. 
Then 
\[ \dim (f(X_1) ) \subset \big ( \bigcup_{(a,b) \in \N} \dim(X) + (-a-b,a,b) \big) \cap \N^3.\] 
\end{theo}
\begin{proof}
Since the dimension is stable under definable bijection, using the Graph of $f$, we can assume that 
$f$ is a coordinate projection, and with an induction, 
we can even assume that $f$ is a coordinate projection along one single variable. 
Hence let us consider $K \to L$ an extension and 
\begin{flalign*}
& \pi_1: L^{m+1} \times \Gamma^n \times k^p \to L^m \times \gamma^n \times k^p &\\
& \pi_2: L^{m} \times \Gamma^{n+1} \times k^p \to L^m \times \gamma^n \times k^p &\\
& \pi_3: L^{m} \times \Gamma^n \times k^{p+1} \to L^m \times \gamma^n \times k^p &
\end{flalign*}
Then, it follows form  definition \ref{defi:d3sortes} that  
if 
$z\in L^{m+1} \times \Gamma^n \times k^p$ 
\[ d(\pi_1(z) /K) = 
\begin{cases}
d(z/K) \ or \\
d(z/K) +(-1,0,0) \ or \\
d(z/K) +(-1,1,0) \ or \\
d(z/K) +(-1,0,1).
\end{cases}
\]
If 
$z\in L^{m} \times \Gamma^{n+1} \times k^p$ 
\[ d(\pi_2(z)/K) = 
\begin{cases}
d(z/K) \ or \\
d(z/K) +(0,-1,0).
\end{cases}
\]
If 
$z\in L^{m} \times \Gamma^n \times k^{p+1}$ 
\[ d(\pi_3(z)/K) = 
\begin{cases}
d(z/K) \ or \\
d(z/K) +(0,0,-1).
\end{cases}
\]
\end{proof}
If one wants to obtain a $\N$-valued dimension one has to consider the following definition. 
\begin{defi}
Let $X \subset  L^{m} \times \Gamma^n \times k^{p}$ be a definable set.
We set 
\[ \dim_\N (X) = \max_{(a,b,c) \in \dim(X) } a+b+c.\] 
\end{defi}

\begin{theo}
If $X \subset K^m$ (resp. $\Gamma^n$, resp. $k^p$) is a definable set, then 
$\dim_\N(X)$ corresponds to the classical dimension as defined in section \ref{sec1.2} 
(resp. as a definable set of a totally ordered divisible group, resp. as a constructible subset of the affine space over $k$). \par 
In addition, if  $X \subset  L^{m} \times \Gamma^n \times k^{p}$, 
$Y\subset  L^{m'} \times \Gamma^{n'} \times k^{p'}$ are definable sets and 
$f : X \to Y$ is a definable map, 
\[\dim_\N(f(X)) \leq \dim_\N (X).\]
\end{theo}

\begin{rem}
It is possible to state the above results in terms of cell. 
Let $(a,b,c) \in \N^3$, and $C_1 \subset K^m \times \Gamma^n$ be a $X-(i_1,\ldots,i_n)$-cell of dimension 
$(a,b)$ (for $X\subset K^m$ some definable set of dimension $a$).
Let $C_2 \subset X \times k^p$ be a definable set such for all $x\in X$, 
$(C_2)_x \subset k^p$ is a locally closed subset of dimension $c$. 
Then we abusively denote by 
\[C = C_1\cap C_2 := \{ (x,\gamma,\lambda) \in  K^{m} \times \Gamma^n \times k^{p} \st 
(x,\gamma) \in C_1 \ \and \ (x,\lambda) \in C_2\}.\]
Then we say that $C$ is a $X$-cell of dimension $(a,b,c)$.
In particular, if $x\in C$, then 
$C_x \subset \Gamma^n \times k^p$ is of the form 
$A\times B$ where $A \subset \Gamma^n$ is a cell of dimension $a$ and 
$B \subset k^p$ is a locally closed subset of dimension $b$. \par 
In this context, if $X \subset   K^{m} \times \Gamma^n \times k^{p}$ is a definable set, 
 $(a,b,c) \in \dim (X)$ if and 
only if $X$ contains a cell of dimension $(a,b,c)$.
\end{rem}

\section{Tropicalization with Berkovich spaces}
\label{sec3}

As we have explained in the introduction, this work started as an attempt to 
generalize the result \cite[th 3.2]{Ducpo}. This is achieved with
\begin{theo}
\label{theoimprovduc}
Let $k$ be a non-Archimedean field. 
Let $X$ be a compact $k$-analytic space, and let 
$f_1, \ldots, f_n$ be $n$ analytic functions defined on $X$. 
This defines a function 
\[\begin{array}{llll}
|f| : & X & \to & (\R_+)^n \\
       & x & \mapsto & (|f_1(x)|, \ldots , |f_n(x)| )
\end{array} \]
Then 
$|f|(X)\cap (\R_+^*)^n$ is a $\Gamma$-rational polyhedral set of $(\R_+^*)^n$ of 
dimension $\leq \dim(X)$, where $\Gamma = \valeur$.
\end{theo}

\begin{proof} 
It is enough to prove the result when $X$ is affinoid, so we will assume that $X$ is an affinoid space. \par 
First let $K$ be an extension of $k$ such that $K$ is algebraically closed and complete, and 
$|K|=\R_+$. The ground field extension functor leads to the commutative diagram:
\[\xymatrix{
X_K \ar[d]^{\pi} \ar[rd]^{|f_K|} &  \\
X \ar[r]^{|f|} & (\R_+)^n
}\]
where $\pi$ is surjective. 
Hence, $|f|(X) = |f_K|(X_K)$. \par 
Secondly let us show that (under these assumptions on $K$), $|f|(X) = |f|(X(K))$.
So let $(x_1,\ldots, x_n) \in |f|(X) \cap (\R_+^*)^n$. 
Then, $U$ is a nonempty affinoid domain of $X$. 
Moreover, since $|K|=\R_+$, $U$ is a strict affinoid domain, so it contains a rigid point, $z$ say 
(thanks to the Nullstellensatz \cite[7.1.2]{BGR}). 
Hence $|f|(z) = (x_1,\ldots, x_n)$, and this shows that $|f|(X) = |f|(X(K))$. \par 
Eventually we are reduced to show that if $X$ is a $k$-affinoid space, 
$k \to K$ is an algebraically closed non-Archimedean extension 
with $|K| = \R_+$, then 
$|f|(X(K)) \cap (\R_+^*)^n$ is a $\Gamma$-polyhedral set of $(\R_+^*)^n$ of 
dimension less or equal than $\dim (X)$. \par 
That $|f|(X(K))$ is a definable set is now just an application of the analytic elimination theorem \cite[3.8.2]{Li_rig}.
Indeed, by definition of an affinoid space, $X$ can be identified with the Zariski closed subset 
of a closed polydisc, that is to say, we may find an integer $N$, and 
$g_1, \ldots ,g_m \in k \langle T_1,\ldots , T_N\rangle$ such that 
$X(K)$ is identified with the Zariski closed subset 
$\{ z=(z_1, \ldots , z_N) \in (K^\circ)^N \ \big| \ g_i(z)=0, \ i=1 \ldots m  \}$. 
Moreover, the functions $f_1, \ldots ,f_n$ of $X$ arise as functions of 
$k \langle T_1,\ldots , T_N\rangle$, so that we will see the $f_i's$ as elements of 
$k \langle T_1,\ldots , T_N\rangle$. 
In addition, the $f_i$'s induce a map 
$f: X \to \mathbb{A}_k^{n,an}$. We can easily reduce to the case 
$f: X \to \mathbb{B}^n$. 
Now, 
\begin{gather*}
f(X(K)) = \\
\{ (x_1, \ldots , x_n ) \in (K^\circ)^n \ \big| \ \exists (z_1, \ldots , z_N) \in (K^\circ)^N 
\ \text{such  that} \ 
g_i(z)=0, \ i=1\ldots m, f_j(z)=x_j, \ j=1 \ldots n \}
\end{gather*}
and this is a subanalytic set of $(K^\circ)^n$ of dimension less or equal than $d$. 
According to theorem \ref{theodimmN}, 
$\Trop(f(X(K)) \cap (\R_+^*)^n$ is then a definable set of $(\R^*_+)^n$ of dimension $\leq d$, hence 
$\Trop(f(X(K)) \cap (\R_+^*)^n$ is also a definable set of $(\R_+^*)^n$ of dimension $\leq d$. \par 
Finally, $\Trop(f(X(K)) \cap (\R_+^*)^n$ is a closed subset of $(\R_+^*)^n$. 
Indeed $X$ is compact, $\Trop \circ f$ is continuous and remind that 
\[\Trop(f(X(K)) = \Trop(f(X)).\]
Hence using the next lemma (after passing to subsets of $\R^n$ with a logarithm map), we 
conclude that $\Trop(f(X(K)) \cap (\R_+^*)^n$ is a polyhedral set of $(\R_+^*)^n$ of dimension 
less or equal than $\dim(X)$.
\end{proof}

\begin{lemme}
Let us consider a nonempty set 
\[P = \bigcap_{j=1}^m \{ (\gamma_i) \in \R^n \st 
\lambda_j + \sum _{i=1}^n a_i^j\gamma_i \bowtie_j 0 \}\]
where $\lambda_j \in \R$,  $a_i^j \in \R$ and $\bowtie_j \in \{ <, \leq \}$.
Then 
\[\overline{P} = \bigcap_{j=1}^m \{ (\gamma_i) \in \R^n \st 
\lambda_j + \sum_{i=1}^n a_i^j \gamma_i \leq 0\}.\]
\end{lemme}
Note that if $P$ was empty, it might happen that the set obtained replacing the $<$ by some $\leq$ 
would become nonempty. 
\begin{proof}
Let us denote by 
\[Q = \bigcap_{j=1}^m \{ (\gamma_i) \in \R^n \st 
\lambda_j + \sum_{i=1}^n a_i^j\gamma_i \leq 0\}.\]
In one hand it is clear that $\overline{P} \subset Q$. \par 
A simple calculation shows that if 
$p\in P$ and $q\in Q$ and $u \in ]0,1]$, then $up + (1-u)q \in P$. 
Let then $q\in Q$ and let us pick $p\in P$. 
When $u \to 0$ we obtain that $q\in \overline{P}$, which shows that 
$Q \subset \overline{P}$.
\end{proof}

\bibliographystyle{alpha}
\bibliography{bibli}

\end{document}